\documentclass[letter,11pt]{article}
\usepackage{amsmath, amsfonts, amssymb, amsthm, amscd, bezier}
\usepackage[brazil,english]{babel}
\usepackage[latin1]{inputenc}
\usepackage[T1]{fontenc}
\usepackage{graphicx,psfrag}
\usepackage{graphicx}
\usepackage[all]{xy}
\usepackage{color}
\usepackage{mathrsfs}

\usepackage{geometry}

\newtheorem{theorem}{Theorem}[section] 
\newtheorem{remark}{Remark}[section] 
\newtheorem{definition}{Definition}[section] 
\newtheorem{proposition}{Proposition}[section] 
\newtheorem{corollary}{Corollary}[section] 
\newtheorem{lemma}{Lemma}[section] 

\newcommand{\cqd}{{{\hfill $\square$}}\vspace{0.3cm}}

\def\href{}

\PassOptionsToPackage{x11names}{xcolor}

\newcommand{\di}{\,\mathrm{dim}\,}
\newcommand{\inte}{\,\mathrm{int}\,}
\newcommand{\ke}{\,\mathrm{ker}\,}
\newcommand{\mo}{\,\mathrm{mod}\,}
\newcommand{\im}{\,\mathrm{Im}\,}

\newcommand{\R}{\mathbb{R}}
\newcommand{\re}{\mathbb{R}}

\newcommand{\N}{\mathbb{N}}
\newcommand{\Z}{{\mathbb{Z}}}

\newcommand{\tcc}{\mathcal{T}}

\newcommand{\me}{{\mathcal{M}}}

\def\dim{\operatorname{dim}}

\def\ss{\mathbb{S}^{2}\times \mathbb{S}^{1}}
\def\T{T^{2}_{g}}

\def\Te{\mathcal{T}_{g}}
\def\Tu{\mathcal{T}_{\widehat{g}}}

\begin{document}

\setlength{\baselineskip}{0.8\baselineskip}
\setlength{\parskip}{0.04\baselineskip}

\title{Smale Flows on $\ss$ }

\author{
 Ketty A. de Rezende\thanks{Partially supported by CNPq under grant 302592/2010-5 and by FAPESP under grant 2012/18780-0.} 
 \and Guido G. E. Ledesma\thanks{Supported by CNPq under grant 141717/2012-2} 
 \and Oziride Manzoli Neto\thanks{ Supported by FAPESP under grant 2012/24249-5.}
 }

\maketitle 
\begin{abstract}
In this paper, we use abstract Lyapunov graphs as a combinatorial tool to obtain a complete classification of Smale flows on $\ss$. This classification gives necessary and sufficient conditions that must be satisfied by an (abstract) Lyapunov graph in order for it to be associated to a Smale flow on $\ss$.
\end{abstract}
\section{Introduction}

In this article, we investigate the relationship between the topology and dynamics of Smale flows on $\ss$. A strategy that has been very successful in analyzing smooth flows on manifolds is to consider first, the local dynamics described by the chain recurrent behavior on the basic sets and secondly, to examine how these basic sets fit together.

Our approach will be qualitative in nature and make use of the methods introduced by Franks, \cite{franks1985nonsingular} and further developed by de Rezende, \cite{de1987smale}, in their study of Smale flows on $\mathbb{S}^3$ and by  Yu in \cite{yu2012} in his study of non-singular Smale flows on $\ss$. Our main contribution herein is to consider Smale flows on $\ss$ with singularities, which adds considerable difficulty to the embedding problem of the basic blocks containing the basic sets.

A qualitative analysis of a flow on a compact manifold $M$ can be divided into two parts, namely, into chain recurrent pieces and and orbits that exhibit gradient-like behaviour among these pieces. This was very nicely captured in a theorem of Conley's \cite{conley1978isolated}  that proves the existence of continuous  Lyapunov functions for continuous flows $\phi_{t}:M\rightarrow M$. Results of Wilson, \cite{wilson1969}, assert that it is possible to choose a smooth Lyapunov function $f:M\to \re$.

\begin{definition}
If $\phi_{t}:M\rightarrow M$ is a flow and $\epsilon>0,$ we say there is an $\epsilon$-chain from $x$ to $y$ provided that there exist points 
$x_{1}=x,x_{2},\dots,x_{n}=y$ and real numbers $t(i)>1$ such that 
$$d(\phi_{t(i)}(x_{i}),x_{i+1})<\epsilon$$
for all $1\leq i<n.$ A point $x$ is called \textbf{chain-recurrent} if for any $\epsilon>0$ there is an $\epsilon$-chain from $x$ to $x$. The set of chain recurrent points 
$\mathfrak{R}$ is called the \textbf{chain recurrent set}.
\end{definition}

It is easy to see that the chain recurrent set is compact and invariant under the flow.

\begin{definition}
If $\phi_{t}:M\rightarrow M$ is a smooth flow, then a smooth function $f:M\rightarrow \re$ will be called a \textbf{Lyapunov function} provided
\begin{enumerate}
\item $\frac{d}{dt}(f(\phi_{t}x))<0$ if it is not in the chain recurrent set $\mathfrak{R}$.
\item If $x,y\in \mathfrak{R}$ then $f(x)=f(y)$ if and only if for each $\epsilon>0$ there are $\epsilon$-chains from $x$ to $y$ and from $y$ to $x$.
\end{enumerate}
\end{definition}

The qualitative study of flows will first be local in nature, that is, it comprises the analysis of the dynamics of the chain recurrent pieces. In this work, for example,  dynamical-topological invariants associated to the basic sets will characterize the possible  basic blocks that contain them as maximal invariant sets. The second part, constitutes a global characterization of how these pieces fit together to form the manifold and realize a smooth flow  on it. 
The question of how the collection of the chain recurrent pieces fit together to form $\ss$  is, in fact, an embedding problem. In this work, as in \cite{de1987smale} and \cite{yu2012} we make use of a Lyapunov graph as a combinatorial tool which gives a global picture of the disposition of the basic sets within the manifold. 

For a Smale flow $\phi_{t}$ on a compact $3$-manifold $M$ with  Lyapunov function $f:M\rightarrow \re$ define the following equivalence relation on $M$: $x\sim_{f} y$ if and only if $x$ and $y$ belong to the same connected component of a level set of $f$. This determines a graph $M/\sim_{f}$ which we say is a \textbf{\textit{Lyapunov graph associated to the  Lyapunov function $f$ and the flow $\phi_{t}$ }} if  the vertices are labelled with a chain recurrent  flow on a compact set and the edges with the genera of the level surfaces given by $f^{-1}(c)$ where $c$ is a regular value. 

One should point out that this definition differs from the ones in \cite{franks1985nonsingular} and in \cite{yu2012} since the dynamics under consideration therein, being non-singular, forced the genera to be equal to one in the first case and less than or equal to one in the latter case. The presence of singularities, as is the case herein, permits level surfaces of all genera. As a guiding principle, the more complicated the flow or if the phase space is higher dimensional, the more labelling with topological invariants  of the associated Lyapunov graph  may be required in order to obtain meaningful characterization results.

The \textbf{\textit{cycle rank}} of the graph is the maximum number of edges that can be removed without disconnecting the graph and denoted by $\beta(L)$. The \textbf{\textit{indegree}} of a vertex $u$ in $L$ is the number of incoming edges incident to $u$ and the \textbf{\textit{outdegree}} of $u$ is the number of outgoing edges incident to $u$.

\begin{definition}
An \textbf{abstract Lyapunov graph} is a finite, connected, oriented graph $L$ such that:
\begin{enumerate}
\item $L$ has no oriented cycles;
\item each vertex is labelled with a chain recurrent flow on a compact space;
\item each edge is labelled with a non-negative number $g$, which we refer to as the weight on the edge.
\end{enumerate} 
\end{definition}

The following result is the main theorem in this paper which characterizes Smale flows on $\ss$. 

\begin{theorem}\label{MainTheorem}
Let $L$ be an abstract Lyapunov graph. $L$ is associated with a Lyapunov function and a Smale flow  $\phi_{t}$ on $\ss$ if and only if:
\begin{enumerate}
\item The underlying graph $L$ is an oriented graph with exactly one edge attached to each sink or source vertex. Moreover, the sink $(\text{source})$ vertex is labelled with an index $0\,(\text{index}\, 3)$ singularity or an attracting $(\text{repelling})$ periodic orbit. 
\item If a vertex of indegree $e^+$ and outdegree $e^-$ is labelled with a singularity of index $2\,\, ($index $1)$, then $1\leq e^{+}\leq 2$ and $e^{-}=1$ $(1\leq e^{-}\leq 2$ and $e^{+}=1)$.
\item If a vertex $v$ of indegree $e^+>0$ and outdegree $e^->0$ is labelled  with a suspension of a subshift of finite type associated to a non-negative integer matrix  $A_{m\times m}$ with $k=\di\ke\left(\left(I-B\right):F^{m}_{2}\rightarrow F^{m}_{2}\right))$, $F_{2}=\Z/2$, $b_{ij}=a_{ij}\mo 2$ and $g_{i}^{+}(g_{j}^{-})$ are the weights on the incoming $($outgoing$)$ edges incident to the vertex $v$, where $G^-=\sum^{e^-}_{i=1}g^-_i$ and  $G^+=\sum^{e^+}_{i=1}g^+_i$, then there are two cases to consider:
$\beta(L)=0$ and $\beta(L)=1$. 
\begin{enumerate}
\item $\beta(L)=1$ we have:
\begin{equation}\label{MainTheorem-eq1}
k+1-G^-\leq e^+\leq k+1 \quad\text{and}\quad k+1-G^+\leq e^-\leq k+1.
\end{equation}

\item $\beta(L)=0$ we have: 
\begin{enumerate}

\item  There exists at most one vertex labelled with a suspension of a subshift of finite type with
\begin{equation}\label{MainTheorem-eq2}
k-G^{-}= e^{+}\quad\text{and}\quad k-G^{+}= e^{-}
\end{equation}
Any other vertex labelled with a suspension of a subshift of finite type satisfies the inequalities in~(\ref{MainTheorem-eq1}).
\item If no vertex on $L$ satisfies the equality in~(\ref{MainTheorem-eq2}) then there must be an edge in $L$ with non zero weight.
\end{enumerate}

\end{enumerate}
\item All vertices must satisfy the Poincar\'e-Hopf condition, i.e., for a vertex labelled with a singularity of index $r$, the condition is 
\begin{equation}\label{MainTheorem-eq3}
(-1)^r=e^+-e^--\sum g^+_j+\sum g^-_i
\end{equation}
and for a vertex labelled with a suspension of a subshift of finite type or a periodic orbit, the condition is 
\begin{equation}\label{MainTheorem-eq4}
0=e^+-e^--\sum g^+_j+\sum g^-_i
\end{equation}
\end{enumerate}
\end{theorem}

The proof of the necessity of the conditions in Theorem~\ref{MainTheorem} follow from Propositions~\ref{pr6}, \ref{pr5}, \ref{pr7} and~\ref{le7}. The proof of the sufficiency of the conditions in Theorem~\ref{MainTheorem} follows from Proposition~\ref{pr-suff}.

In Section~\ref{sec2} we introduce  background material. In Section~\ref{sec3} we provide some topological properties of embedded surfaces in $\ss$. We introduce the notion of $3$-manifolds of handlebody type. In Section~\ref{sec4} we study the relationship between the graph and manifolds of handlebody type. In Sections~\ref{sec5} and~\ref{sec-suff} we state and prove propositions which comprise the foundation of the proof of the main result, Theorem~\ref{MainTheorem}. Throughout this paper, we consider homology with $F_2=\Z_2$ coefficients.


\section{Background and Definitions}\label{sec2}
In this section, we provide background material, basic definitions and theorems, essential to the understanding of our work. 

\subsection{Smale Flows}

A compact invariant set $\Lambda$ for a smooth flow $\phi_{t}:M\rightarrow M$ is said to have a \textbf{\textit{hyperbolic structure}} provided that the tangent bundle of $M$ restricted to $\Lambda$ can be written as the Whitney sum of three sub-bundles $E^{u}+E^{s}+E^{c}$, each being invariant under $D\phi_{t}$ for all $t$, in such a way that $E^{c}$ is spanned by the vector field tangent to the flow and there are constants $C$, $\alpha>0$ satisfying
  
$$\left\|D\phi_{t}(v)\right\|\leq Ce^{-\alpha t}\left\|v\right\|\quad\text{for $v\in E^{s}$, $t\geq 0$}$$ 
and 
$$\left\|D\phi_{t}(v)\right\|\geq C^{-1}e^{\alpha t}\left\|v\right\|\quad\text{for $v\in E^{u}$, $t\geq 0$}.$$ 
An important consequence of the hyperbolicity is that $\mathfrak{R}$ is decomposed into finitely many irreducible pieces \cite{smale1967}.
\begin{theorem}
Suppose that the chain recurrent set $\mathfrak{R}$ of a flow on a compact manifold $\phi_{t}:M\rightarrow M$ has a hyperbolic structure. Then $\mathfrak{R}$ is a finite disjoint union of compact invariant sets $\Lambda_1,\dots,\Lambda_n$ and each $\Lambda_i$ contains a point whose forward orbit is dense in $\Lambda_i$. 
\end{theorem}
The sets $\Lambda_i$ are called \textbf{\textit{basic sets}} of the flow and are precisely the chain transitive pieces of $\mathfrak{R}$, which are not separated by Lyapunov function.\\
If $\Lambda$ is a compact invariant hyperbolic set for a flow then each orbit in $\Lambda$ has a \textbf{\textit{stable}} and \textbf{\textit{unstable manifold}}. These are defined as follows for $x\in\Lambda$,
$$W^{s}(x)=\left\{y\in M\,|\,\text{for some}\,r\in\re,\,d(\phi_{t}(y),\phi_{t+r}(x))\rightarrow 0\,\text{as}\,t\rightarrow \infty\right\}$$
and
$$W^{u}(x)=\left\{y\in M\,|\,\text{for some}\,r\in\re,\,d(\phi_{t}(y),\phi_{t+r}(x))\rightarrow 0\,\text{as}\,t\rightarrow -\infty\right\}$$
\begin{definition}
A flow $\phi_{t}:M\rightarrow M$ with hyperbolic chain recurrent set $\mathfrak{R}$ is said to satisfy the \textbf{transversality condition} provided that for each $x,y\in\mathfrak{R}$, the manifolds $W^{s}(x)$ and $W^{u}(y)$ intersect transversally. 
\end{definition}
\begin{definition} A smooth flow $\phi_{t}:M\rightarrow M$ on a compact manifold is called a \textbf{Smale flow} provided:
\begin{enumerate}
\item its chain recurrent set $\mathfrak{R}$ has a hyperbolic structure and $\dim\mathfrak{R}\leq 1$;
\item it satisfies the transversality condition.
\end{enumerate}
\end{definition} 

Bowen~\cite{bowen1972} gives a complete dynamical description of the basic sets of Smale flows. 
\begin{theorem}
Let $\phi_{t}$ be a flow with hyperbolic chain recurrent set and $\Lambda$ a basic set of dimension 1, then $\phi_{t}$ restricted to $\Lambda$ is topologically conjugate to the suspension of a subshift of finite type associated to an irreducible matrix. 
\end{theorem}
Hence, the chain recurrent set of a Smale flow is made up of singularities, periodic orbits and suspensions of subshifts of finite type. 
\begin{definition}
If $A$ and $B$ are non-negative integer matrices they are \textbf{flow equivalent} provided the suspension of the subshifts of finite type $\sigma(A)$ and $\sigma(B)$ are topologically equivalent. 
\end{definition}


Also, we will need to define the genus of a manifold and use the following result which characterizes it.

\begin{definition}
Let $M$ be a smooth, compact, connected $n$-manifold with boundary. The \textbf{genus of manifold} $M$, $g(M)$ is the maximal number of mutually disjoint, smooth, compact, connected, two-sided codimension one submanifolds that do not disconnect $M$.    
\end{definition}
The following result gives the relation between the cycle rank of $L$ and the genus of the fundamental group of $M$.
\begin{proposition}\label{co}
Let $M$ be a connected, closed smooth $n$-manifold. Let 
$\phi_{t}$ be a smooth flow on $M$ with associated Lyapunov function $f$. Let $L$ be a Lyapunov graph associated to $f$. Then
$$\beta(L)\leq g(M)$$
\end{proposition}  
See \cite{cruz1998cycle} for the proof.

Let $\phi_t$ be a smooth flow on $\ss$ with $L$ its associated Lyapunov graph. Since, $g(\ss)=1$ by Proposition~\ref{co}, one has
$$
\beta(L)\leq 1
$$ 
Thus, Lyapunov graphs associated  with smooth flows, in particular Smale flows, on $\ss$ have at most one cycle.


\subsection{Filtration and Homology}

The existence of a smooth Lyapunov function for a flow $\phi_t$ implies the existence of a filtration associated to it. 

\begin{definition} If $\phi_{t}:M\rightarrow M$ is a flow with hyperbolic chain recurrent set with basic sets $\left\{\Lambda_{i}\right\}\,(i=1,\dots,n)$, a \textbf{filtration} associated with $\phi_{t}$ is a collection of submanifolds
$M_{0}\subset M_{1}\subset\dots\subset M_{n}=M$ such that
\begin{enumerate}
\item $\phi_{t}(M_{i})\subset \inte M_i$, for each $t>0$;
\item $\Lambda_{i}=\bigcap^{\infty}_{t=-\infty}\phi_{t}(M_{i}-M_{i-1}).$
\end{enumerate}
\end{definition}

The following is a result due to Bowen and Franks~\cite{bowen1977homology}.

\begin{theorem}\label{teor1}
Let $\phi_{t}$ be a Smale flow and let $M_{i},\,i=1,\dots,n$ be a filtration associated to $\phi_{t}$. Suppose that $\Lambda_{i}=\bigcap^{\infty}_{t=-\infty}\phi_{t}(M_{i}-M_{i-1})$ is a basic set of index $k$ labelled with a structure matrix $A_{m\times m}$. Then
$$H_{s}(M_{i},M_{i-1}; F_{2})=0,\,\text{se}\quad s\neq k,k+1;$$
$$H_{k}(M_{i},M_{i-1}; F_{2})\cong F_{2}^{m}/(I-B)F_{2}^{m};$$
$$H_{k+1}(M_{i},M_{i-1}; F_{2})\cong \ke((I-B)\,\text{on}\,\,F_{2}^{m}),$$
where $B=A\mo 2$.
\end{theorem}

See \cite{franks1982} for the proof.

The following theorem in \cite{de1987smale} for Smale flows on $S^{3}$, completely classifies Lyapunov graphs for flows on $S^{3}$.

\begin{theorem}\label{teorem1}
Let $L$ be an abstract Lyapunov graph. $L$ is associated with a Smale flow $\phi_t$ on $\mathbb{S}^3$ if only if the following conditions hold:
\begin{enumerate}
\item The underlying graph $L$ is a tree with exactly one edge attached to each sink or source vertex. Moreover, the sink (source) vertex is labelled with an index $0$ (index $3$) singularity or an attracting (repelling) periodic orbit. 
\item If a vertex is labelled with a singularity of index $2$ (index $1$) then $1\leq e^+\leq 2$ and $e^-=1$ ($e^+=1$ and $1\leq e^-\leq 2$).
If a vertex is labelled with a suspension of a subshift of finite type and $A_{m\times m}$ is the non-negative integer matrix representing this subshift, then   
\begin{center}
$e^{+},\,e^{-}>0,$\\
$k+1-G^-\leq e^{+}\leq k+1,\,\text{with}\quad G^-=\sum^{e^-}_{i=1}g^-_i\quad \text{and}$\\
$k+1-G^+\leq e^{-}\leq k+1\,\text{with}\quad G^+=\sum^{e^+}_{j=1}g^+_j$
\end{center}
where $k=\di\ke((I-B):F^m_2\rightarrow F^m_2)$ and $b_{ij}=a_{ij}\mo 2$, $e^+(e^-)$ is the number of incoming (outgoing) edges incident to the vertex and $g^+_j(g^-_i)$ is the weight on an incoming (outgoing) edge incident to the vertex.
\item All vertices satisfy the Poincar\'e-Hopf condition (\ref{MainTheorem-eq3}) and (\ref{MainTheorem-eq4}).
\end{enumerate}    
\end{theorem}

\section{Surfaces Embedded in $\ss$}\label{sec3}

\subsection{Some topological facts in $\ss$}
In our work it is important to have a characterization of the generators of the two dimensional homology of $\ss$.
\begin{lemma}\label{le16}
Let $S$ be an orientable closed connected surface in $\ss$ such that $S$ is non separating and $i:S\rightarrow\ss$ is the inclusion map. Then the induced homomorphism  
\begin{center}
$i_\ast: H_2(S)\rightarrow H_2(\ss)$
\end{center} 
is an isomorphism.
\end{lemma}
\begin{proof}
Since $S$ is non separating, there are two points $A$ and $B$ on a tubular neighborhood of $S$ which can be connected by a path that intersects $S$ in one single point $C$. This path can be extended to a loop which intersects $S$ in $C$. It follows that the one dimensional homology class $[\gamma]$ and the two dimensional homology class $[S]$ have a non-zero intersection number. Therefore, both $[\gamma]$ and $[S]$ are non-trivial elements in the homology of $\ss$. Since, $H_2(\ss)\cong F_2$, one has that $i_\ast$ is an isomorphism. 
\end{proof}

\begin{definition} Let $S$ be a closed surface in $\ss$ and let $i:S\rightarrow \ss$ be the inclusion embedding. Let
\begin{center}
$i_\ast: \pi_1(S)\rightarrow \pi_1(\ss)$
\end{center}
 be the homomorphism induced by $i$. We say that $S$ is $\pi_1$-\textbf{trivial} in $\ss$ if $i_\ast$ is trivial and $\pi_1$-\textbf{non-trivial} otherwise.  
\end{definition}

The study of embedded surfaces in $\ss$ is essential to our work. Some of these embeddings have been previously studied, such as, the embedding of the sphere $\mathbb{S}^{2}$ in \cite{bonatti2000knots} and the embedding of the torus in  \cite{yu2012}. In this section, we present other embeddings that are relevant for our work. The following result of the embedding of $\mathbb{S}^{2}$ in $\ss$ is a well known result with a nice proof  in \cite{beguin2002flots}. 
 
\begin{proposition}\label{pr1}
Let $S$ be a differential embedding of the $2$-sphere in $\ss$, then
\begin{enumerate}
\item either $S$ bounds a $3$-ball,
\item or $S$ is homotopic to a fiber $S^{2}\times \left\{t\right\}$, with $t\in S^{1}$.
\end{enumerate}
\end{proposition}
Now let $\phi_{t}$ be a Smale flow on $\ss$ with Lyapunov function $f:\ss\rightarrow \R$ and let $L$ be the Lyapunov graph associated to $f$. Let $c\in\re$ be a regular value, then $f^{-1}(c)$ is the disjoint collection of orientable, closed connected surfaces. Let $\Te\subset f^{-1}(c)$ be a connected, closed surface of genus $g$. Then by the topology of $\ss$, $\Te$ can be  a non separating surface in $\ss$.

If $\Te$ is non separating, there exists another connected component of the regular level set $\Tu$ non-parallel to $\Te$, such that $\ss-(\Te\sqcup\Tu)$ has two components. We denote their closures by $M_1$ and $M_2$. Therefore, $M_1$ and $M_2$ satisfy 
\begin{center}
$M_1\cup M_2=\ss\quad\text{and}\quad M_1\cap M_2=\Te\sqcup\Tu$
\end{center}
Also, $\partial M_1=\partial M_2=\Te\sqcup\Tu$. The following lemma gives us homological information on $M_1$ and $M_2$.
\begin{lemma}\label{le14}
$M_1$ and $M_2$ as defined above satisfy: 
\begin{enumerate}
\item $H_2(M_i)\cong F_2$,
\item $\di H_1(M_i)=\widehat{g}+g$.
\end{enumerate}
\end{lemma}
\begin{proof}
 We consider the Mayer-Vietoris exact sequence of the pair 
$(M_1,M_2)$,
\begin{equation*}\label{eq33}
0\rightarrow H_{3}(M_{1})\oplus H_{3}(M_{2})\rightarrow H_{3}(\ss)\rightarrow H_{2}(M_{1}\cap M_{2})\rightarrow H_{2}(M_{1})\oplus H_{2}(M_{2})\stackrel{\alpha}{\rightarrow} H_{2}(\ss)\rightarrow \ldots
\end{equation*}
 Since, $M_1$ and $M_2$ are compact manifolds with boundary, then $H_{3}(M_1)=0=H_{3}(M_2)$. On the other hand, $H_{3}(\ss)\cong F_2\cong H_{2}(\ss)$ and as $M_1\cap M_2$ has two components, then $H_2(M_1\cap M_2)\cong F_2^2$. Therefore, one has
\begin{equation}\label{eq14}
\di H_2(M_1)+\di H_2(M_2)\leq 2.
\end{equation}
On the other hand, we consider the exact sequence of the pair $(M_i,\partial M_i)$,
\begin{equation}\label{eq34}
0\rightarrow H_3(M_i,\partial M_i)\rightarrow H_2(\partial M_i)\rightarrow H_2(M_i)\rightarrow H_2(M_{i},\partial M_i).
\end{equation}
Since, $M_i$ is a $3$-manifold with boundary, then $H_{3}(M_i,\partial M_i)\cong F_2$. Also $\partial M_i$ has two components, and hence $H_2(\partial M_i)\cong F_2^2$. Therefore, by exactness
\begin{equation}\label{eq15}
\dim H_2(M_i)\geq 1.
\end{equation}
By (\ref{eq14}) and (\ref{eq15}), one concludes that
\begin{center}
$H_{2}(M_i)\cong F_2.$
\end{center}
So, we obtain the following exact sequence from (\ref{eq34})
\begin{equation*}
0\rightarrow H_2(M_i,\partial M_i) \rightarrow H_1(\partial M_i)\rightarrow H_1(M_i)\rightarrow H_1(M_i,\partial M_i)\rightarrow \widetilde{H}_{0}(\partial M_i)\rightarrow 0.
\end{equation*}
Therefore, by exactness
\begin{equation}\label{eq16}
\di H_1(M_i)=\di H_1(\partial M_i)-\di H_2(M_i,\partial M_i)+\di H_1(M_i,\partial M_i)-\di \widetilde{H}_0(\partial M_i).
\end{equation}
At this point, one needs to consider the following proposition, which follows from the universal coefficient theorem.

\begin{proposition}\label{pr2}
Let $(X,A)$ be a pair of topological spaces, and let $F$ be a field. There is a natural isomorphism
$$\beta:H^n(X,A;F)\rightarrow Hom_{F}(H_n(X,A;F),F)\cong H_n(X,A;F)^\ast$$
\end{proposition}

Since, $\partial M_i\cong\Te\sqcup\Tu$, then $\dim H_1(\partial M_i)=2g+2\hat{g}$ and $\widetilde{H}_0(\partial M_i)\cong F_2$. On the other hand, by Proposition~\ref{pr2} and by the Poincar\'e duality, one has that
\begin{center}
$H_1(M_i,\partial M_i)\cong (H_2(M_i))^\ast\quad\text{and}\quad H_2(M_i,\partial M_i)\cong\left(H_{1}(M_i)\right)^\ast$.
\end{center}
In $(\ref{eq16})$ one has
\begin{equation}
\di H_1(M_i)=g+\hat{g}.
\end{equation} 
\end{proof}
If $\Te$ is separating, this embedding splits 
$\ss$ in two submanifolds $M_{1}$ and $M_{2}$ such that:
\begin{center} 
$\ss = M_{1}\cup M_{2}\quad\text{and}\quad M_{1}\cap M_{2}=\Te. $
\end{center}
\begin{lemma}\label{le15}
Let $M_1$ and $M_2$ be defined as above. Then
\begin{enumerate}
\item $\di H_2(M_1)+\di H_2(M_2)\leq 1$,
\item If $H_2(M_1)=0$, then $\di H_1(M)=g$,
\item If $H_2(M_1)\cong F_2$, then $\di H_1(M)=g+1$.
\end{enumerate}
Analogously for $M_2$.
\end{lemma}
\begin{proof}
Consider the Mayer-Vietoris exact sequence of $M_1$ and $M_2$. 
\begin{equation}\label{seq9}
0\rightarrow H_{3}(M_{1})\oplus H_{3}(M_{2})\rightarrow H_{3}(\ss)\rightarrow H_{2}(\T)\rightarrow H_{2}(M_{1})\oplus H_{2}(M_{2})\rightarrow H_{2}(\ss)\rightarrow 
\end{equation}
 Since, $M_1$ and $M_2$ are compact manifolds with boundary, then $H_3(M_1)=0= H_3(M_2)$. Also, $M_1\cap M_2$ has one component, hence $H_2(M_1\cap M_2)\cong F_2$. Therefore, by the exactness of sequence $(\ref{seq9})$ one has
\begin{equation}\label{eq4}
\di H_{2}(M_{1})+\di H_{2}(M_{2})\leq 1.
\end{equation}
Hence, there are the following possibilities: 
\begin{enumerate}
\item[(a)] either $H_{2}(M_{1})=0$ and $H_{2}(M_{2})=0$ 
\item[(b)] or $H_{2}(M_{1})\cong F_2$ and $H_{2}(M_{2})=0$
\item[(c)] or $H_{2}(M_{1})=0$ and $H_{2}(M_{2})\cong F_2$.
\end{enumerate}
Now consider the exact sequence of the pair $(M_i,\partial M_i)$, 
\begin{equation}\label{seq10}
0\rightarrow H_3(M_i,\partial M_i)\rightarrow H_2(\partial M_i)\rightarrow H_2(M_i)\rightarrow H_2(M_i,\partial M_i)\rightarrow \ldots
\end{equation}
Since, $M_i$ is a $3$-manifold with boundary $H_3(M_i,\partial M_i)\cong F_2$. Also, $\partial M_i$ has one component, hence $H_2(\partial M_i)\cong F_2$. Therefore,
$H_3(M_i,\partial M_i)$ is isomorphic to $H_2(\partial M_i)$. 
Thus, we can write the exact sequence $(\ref{seq10})$ as
\begin{equation*}
0\rightarrow H_2(M_i)\rightarrow H_2(M_i,\partial M_i) \rightarrow H_1(\partial M_i)\rightarrow H_1(M_i)\rightarrow H_1(M_i,\partial M_i)\rightarrow  0
\end{equation*}
which implies that 
\begin{equation}\label{eq1}
2g=\di H_{2}(M_i,\partial M_i)-\di H_{2}(M_{i})+\di H_1(M_i)-\di H_1(M_i,\partial M_i).
\end{equation}
By Proposition~\ref{pr2} and by Poincar\'e duality, one has that
\begin{center}
$H_1(M_i,\partial M_i)\cong \left(H_{2}(M_i)\right)^\ast \cong \left(H_{2}(M_i)\right)\quad\text{and}\quad H_2(M_i,\partial M_i)\cong\left(H_{1}(M_i)\right)^\ast \cong \left(H_{1}(M_i)\right).$
\end{center}
Hence, equation $(\ref{eq1})$ can be written as
\begin{equation}\label{eq2}
g=\di H_{1}(M_{i})-\di H_{2}(M_{i}).
\end{equation}
Now, one concludes the proof with a case analysis. 
 
\textbf{Case $(a)$:} Suppose $H_2(M_1)=0$ and $H_2(M_2)=0$.

 Hence, by equation $(\ref{eq2})$ one has for $i=1,2$
\begin{center}
$g=\dim H_{1}(M_{i}).$
\end{center} 

\textbf{Case $(b)$:} Suppose $H_2(M_1)\cong F_2$ and $H_2(M_2)=0$.

\begin{enumerate}
\item[(i)] By the prior case, if $H_{2}(M_{2})=0$ then $H_{1}(M_{2})\cong F^g_2$ hence,  $\di H_1(M_2)=g$
\item[(ii)]In the case $H_{2}(M_{1})\cong F_2$,
equation $(\ref{eq2})$ is written as

$g=\dim H_{1}(M_{1})-1$,

which implies that $\di H_1(M_1)=g+1$.

\end{enumerate}
 
\textbf{Case $(c)$:} Suppose $H_2(M_1)=0$ and $H_2(M_2)\cong F_2$.

In a similar fashion to the previous case one has that
\begin{enumerate}
\item[(i)] if $H_2(M_1)=0$, then $\di H_1(M_1)=g$,
\item[(ii)] if $H_2(M_2)\cong F_2$, then $\di H_1(M_2)=g+1$.
\end{enumerate}
\end{proof}
\begin{definition}
A manifold $M^{3}$ is of \textbf{handlebody} \textbf{type} whenever $\partial M^3$ is homeomorphic to a closed and orientable surface of genus $g$, for some $g\geq 0$ and verifies
$$H_2(M^3)=0\quad\text{and}\quad \di H_1(M^3)=g.$$
The positive number $g$ is defined as the genus of the manifold of handlebody type. 
\end{definition}

\begin{remark}\label{ob1}
By inequality $(\ref{eq4})$ and Lemma~\ref{le15}, at least one of $M_{1}$ or $M_{2}$ is of handlebody type.
\end{remark}
\begin{lemma}\label{le1} 
With the above notation, $M_{2}$ is not a manifold of handlebody type if and only if
\begin{center}
$H_2(M_2)\cong F_2.$
\end{center}                         
\end{lemma} 
\begin{proof} 
$ (\Rightarrow)$ By inequality $(\ref{eq4})$ one has that $\di H_{2}(M_{2})\leq 1$, hence we have two possibilities. If $H_{2}(M_{2})= 0$, by Case $(b)$ in the proof of Lemma~\ref{le15}, it follows that $\di H_1(M_2)=g$. Hence, $M_2$ is of handlebody type contradicting the hypothesis. Therefore, $H_2(M_1)\cong F_2$.

$(\Leftarrow)$ Since, $H_2(M_2)\cong F_2$ then by definition $M_2$ is not of handlebody type.
\end{proof} 

\begin{lemma}\label{le13}
Let $N$ be a $3$-submanifold of $\ss$, such that $N$ is of handlebody type and $\partial N$ is a $\pi_1$-trivial surface in $\ss$. Then the inclusion $j:N\rightarrow \ss$ induces a homomorphism
\begin{center}
$j_\ast:H_1(N)\rightarrow H_1(\ss)$
\end{center}
which is trivial.
\end{lemma}
\begin{proof}
Consider the following diagram
\begin{equation*}
\label{eq:diagram}
  \xymatrix@R+2em@C+2em{
  \partial N \ar[rd]^-k \ar[d]_i &  \\
  N \ar[r]^-j  & \ss
  }
 \end{equation*}
where $i,j$ and $k$ are inclusions. These functions induce the following diagram.
\begin{equation*}
\label{eq:diagram}
  \xymatrix@R+2em@C+2em{
  H_1(\partial N) \ar[rd]^-{k_\ast} \ar[d]_{i_\ast} &  \\
  H_1(N) \ar[r]^-{j_\ast}  & H_1(\ss)
  }
 \end{equation*}
Since, $\partial N$ is $\pi_1$-trivial, then $k_\ast$ is trivial. Thus, in order to prove that $j_\ast$ is trivial, it is only necessary to prove that $i_\ast$ is surjective. Now, consider the exact homology sequence of the pair $(N,\partial N)$,
\begin{equation*}
0\rightarrow H_2(N,\partial N)\rightarrow H_1(\partial N)\stackrel{i_\ast}{\rightarrow}H_1(N)\rightarrow H_1(N,\partial N)\rightarrow \ldots
\end{equation*}
By Proposition~\ref{pr2} and the Poincar\'e duality, one has  
$$
H_1(N,\partial N)\cong (H_2(N))^\ast.
$$
On the other hand, one has that $H_2(N)=0$. Hence the lemma follows.

\end{proof}

\begin{lemma}\label{le4}
Let $\Te$ be a separating connected component of a regular level set associated to some Smale flow $\phi_t$ on $\ss$. Let $M_{1}$ and $M_{2}$ be the closure of the two connected components of $\ss-\Te$. Hence,
\begin{enumerate}
\item if  $\Te$ is $\pi_1$-trivial in $\ss$, then  
either $M_{1}$ or $M_{2}$ is of handlebody type;
\item if $\Te$ is $\pi_1$-non-trivial in $\ss$, then $M_1$ and $M_2$ are of handlebody type. 
\end{enumerate}
\end{lemma}
\begin{proof}
\begin{enumerate}
\item It follows by Remark~\ref{ob1}, that at least one of the submanifolds $M_{1}$ or $M_{2}$ is of handlebody type. 
Suppose that, $M_1$ is a manifold of handlebody type. By Lefschetz Duality, one has that
\begin{equation}\label{eq31}
H_1(\ss,M_1)\cong H_2(M_2).
\end{equation}
Now, consider the homology exact sequence of the pair $(\ss, M_1)$, 
\begin{equation*}
\rightarrow H_2(\ss,M_1)\rightarrow H_1(M_1)\stackrel{\eta}\rightarrow H_1(\ss)\rightarrow H_1(\ss,M_1)\rightarrow 0.
\end{equation*}
Since, $\Te$ is $\pi_1$-trivial in $\ss$, by Lemma~\ref{le13}, one has that $\eta$ is equal to zero. Therefore, $H_1(\ss)\cong H_1(\ss,M_1)$ and by $(\ref{eq31})$ one has that
\begin{center}
$H_2(M_2)\cong F_2.$
\end{center}
The result now follows by Lemma~\ref{le1}.
\item If $\Te$ is $\pi_1$-non-trivial in $\ss$ there exists a simple closed curve $\beta$ in $\Te$ such that $[\beta]\in \pi_{1}(\ss)$ is non-trivial. Now consider the Mayer-Vietoris exact homology sequence.
\begin{equation}\label{seq11}
0\rightarrow H_2(M_1)\oplus H_2(M_2)\rightarrow H_2(\ss)\stackrel{\Delta}{\rightarrow}H_1(X\cap Y)\rightarrow 
\end{equation}

One has that $\Delta$ is non zero, otherwise we would have that the image of $\Delta$ in $H_1(\Te)$ is homologous to a curve $\gamma$ that separates $\Te$ and the intersection number of $\beta$ and $\gamma$ would be zero, since $[\gamma]=0$ in $\Te$. However, since this intersection is $n\ne 0$, we have a contradiction and hence, $\Delta$ is not zero. Thus, by (\ref{seq11}) one has that $H_2(M_1)$ and $H_2(M_2)$ are zero. Thus $M_1$ and $M_2$ are of handlebody type. 
\end{enumerate}
\end{proof}

\section{Manifolds of Handlebody Type and Lyapunov Graphs}\label{sec4}

In this section, we study the relation between Lyapunov graphs and manifolds of handlebody type.
Let $\phi_{t}$ be a Smale flow on $\ss$ with Lyapunov function $f$ and let $L$ be the associated Lyapunov graph. Suppose $L$ is a tree. Choose a regular level set which is a separating surface $\Te$ which splits $\ss$ in two $3$-submanifolds, $\me^{-}_{\Te}$ and $\me^{+}_{\Te}$, and such that the flow is transversal and outward going on $\me^{-}_{\Te}$ and inward going on $\me^{+}_{\Te}$.
\begin{remark}\label{ob2}
By Lemma $\ref{le4}$ one has that:
\begin{itemize}
\item if $\Te$ is  $\pi_{1}$-trivial, then either $\me^{+}_{\Te}$ or $\me^{-}_{\Te}$ is of handlebody type;
\item if $\Te$ is $\pi_{1}$-non-trivial, then $\me^{+}_{\Te}$ and $\me^{-}_{\Te}$ are of handlebody type;
\item if we cut $\ss$ along a separating regular level set, this corresponds to disconnecting the graph $L$  into two subgraphs with dangling edges that represent each one of the manifolds $\me^{+}_{\Te}$ and $\me^{-}_{\Te}$.  
\end{itemize} 
\end{remark} 
\begin{proposition}\label{co2}Let $u$ be a vertex in $L$ and let $e$ be an incoming edge incident to $u$. Suppose that there exists a vertex $v$ labelled with a basic set $\Lambda$ such that $\Lambda$ is contained in $\me^-_{\tcc_g}$. Then there exists an incoming (outgoing) edge with weight $\widehat{g}$ of $v$ such that $\me^+_{\tcc_{\widehat{g}}}\,(\me^-_{\tcc_{\widehat{g}}})$ is contained in $\me^-_{\tcc_g}$.
\end{proposition}
\begin{proof}
Since $L$ is a tree, there exists an incoming or outcoming edge $e_1$ (with weight $\widehat{g}$) of $v$ such that cutting along this edge separates $u$ and $v$ in distinct subgraphs. First suppose that $e_1$ is an incoming edge incident to $v$. Hence, if we cut $\ss$ along $\Tu$ one has $\me^+_{\tcc_{\widehat{g}}}\subset \me^-_{\tcc_{g}}$. On the other hand, if $e_1$ is an outgoing edge incident to $v$, $\me^-_{\tcc_{\widehat{g}}}\subset \me^-_{\tcc_g}$. 
\end{proof}

We have a similar result if the edge $e$ incident to $u$ is outgoing. In what follows other interesting properties of $L$ are determined.
Given a vertex $v$ of the graph $L$ labelled with a basic set $\Lambda_{v}$, each edge point represents a regular level set which is surface of genus $g^{+}_{i}\,(g^{-}_{j})$ embedded in $\ss$ denoted by $\tcc_{g^{+}_{i}}\,(\tcc_{g^{-}_{j}})$. The incoming (outgoing) edges are labelled with weights corresponding to $g^{+}_{i}$'s $(g^{-}_{i})$'s.

\begin{corollary}\label{le5} Let $v$ be a vertex on $L$ and with the above notation.
\begin{enumerate}
\item For $\alpha\in \left\{1,\dots,e^{-}_{v}\right\}$ fixed, one has
\begin{center}
$\me^-_{\tcc_{g^+_k}}\subset \me^-_{\tcc_{g^-_\alpha}}\quad\text{and}\quad\me^+_{\tcc_{g^-_j}}\subset \me^-_{\tcc_{g^-_\alpha}}$
\end{center}
with $k\in\left\{1,\dots,e^{+}\right\}$ and $j\in\left\{1,\dots,e^{-}\right\}-\left\{\alpha\right\}.$
\item For $\beta\in \left\{1,\dots,e^{+}\right\}$ fixed, one has
\begin{center}
$\me^-_{\tcc_{g^+_k}}\subset \me^+_{\tcc_{g^+_\beta}}\quad\text{and}\quad\me^+_{\tcc_{g^-_j}}\subset \me^+_{\tcc_{g^+_\beta}}$
\end{center}
with $k\in\left\{1,\dots,e^{+}\right\}-\left\{\beta\right\}$ and $j\in\left\{1,\dots,e^{-}\right\}.$
\end{enumerate}
\end{corollary}

\begin{proof}   
\begin{enumerate}
\item Since $\tcc_{g^-_\alpha}$ corresponds to an outgoing edge incident to $v$, let $u$ be the vertex on which the edge is incoming. Note that each $\tcc_{g^+_k}$ corresponds
to an outgoing edge incident to vertices $v_k$ that are incoming and incident to $v$. Hence, by the Lemma~\ref{le4} it follows that
\begin{center}
$\me^-_{\tcc_{g^+_k}}\subset \me^-_{\tcc_{g^-_\alpha}}.$
\end{center} 
Analogously, each $\tcc_{g^{-}_{j}}$ corresponds to an outgoing edge incident to vertices  $v_{j}$ and by Lemma~\ref{le4} it follows that
\begin{center}
$\me^+_{\tcc_{g^-_j}}\subset \me^-_{\tcc_{g^-_\alpha}}.$
\end{center}
\item To prove $2$, it suffices to reverse the orientation on the graph. 
\end{enumerate}
\end{proof}
\begin{corollary}\label{co1}
\begin{enumerate} Let $v$ be a vertex of $L$ and with the notation above, one has
\item if $\me^-_{\tcc_{g^-_\alpha}}$ is of handlebody type for $\alpha\in \left\{1,\dots,e^{-}\right\}$, then
\begin{center}
$\me^-_{\tcc_{g^+_k}}\quad\text{and}\quad\me^+_{\tcc_{g^-_j}}$
\end{center}
are of handlebody type for $k\in\left\{1,\dots,e^{+}\right\}$ and $j\in\left\{1,\dots,e^{-}\right\}-\left\{\alpha\right\}$;
\item if $\me^+_{\tcc_{g^+_\beta}}$ is of handlebody type for $\beta\in \left\{1,\dots,e^{+}\right\}$, then
\begin{center}
$\me^-_{\tcc_{g^+_k}}\quad\text{and}\quad\me^+_{\tcc_{g^-_j}}$
\end{center}
are of handlebody type for $k\in\left\{1,\dots,e^{+}\right\}-\left\{\beta\right\}$ and $j\in\left\{1,\dots,e^{-}\right\}$. 
\end{enumerate}
\end{corollary}
\begin{proof} This follows directly from Remark $\ref{ob2}$, of Corollary $\ref{le5}$ and from the fact that a submanifold with a surface boundary of a manifold of handlebody type is of handlebody type. 
\end{proof}
\section{Necessary Conditions on Lyapunov Graphs associated to Smale Flows}\label{sec5}

In this section we will prove a series of propositions that will comprise the proof of the necessity of the conditions stated in the main theorem of this work, Theorem~\ref{MainTheorem}.

Let $\phi_{t}$ be a Smale flow on $\ss$ with Lyapunov function $f:\ss\rightarrow \re$. Let $\Lambda$ be a basic set of $\phi_{t}$ and $f(\Lambda)=c$. Choose $\epsilon>0$ sufficiently small so that $c$ is the only critical value in $[c-\epsilon,c+\epsilon]$. Let $X_0=f^{-1}(-\infty,c+\epsilon]$ and $Z_0=f^{-1}(-\infty,c-\epsilon]$. 

\subsection{Lyapunov Graph with a Cycle}

Let $U$ be the closure of the component of $X_0-Z_0$ which contains $\Lambda$. Then $U$ is a neighborhood of $\Lambda$ whose boundary consists of closed orientable surfaces to which the flow is transverse. The flow enters  $e^+$ of these surfaces and exits the remaining $e^-$. 

\begin{proposition}\label{pr6}
Let  $L$ be a Lyapunov graph associated with a Smale flow  $\phi_{t}$  and a Lyapunov function on $\ss$, such that $\beta(L)=1$. Let $v$ be a vertex of $L$ on the cycle, labelled with a suspension of a subshift of finite type and $A_{m\times m}$ is the non-negative integer matrix representing this subshift. Let $k=\di \ke\overline{I-A}:F_2^m\rightarrow F_2^m$ where  $\overline{I-A}$ is the $\mo 2$ reduction of $I-A$. Then, if $e^+$ and $e^-$ are respectively the indegree and outdegree of $v$, one has:
\begin{center}
$e^+ ,e^- >0,$\\
$e^+\leq k+1,$\\
$e^-\leq k+1,$\\
$k+1-G^-\leq e^{+},\,\text{with}\quad G^-=\sum^{e^-}_{i=1}g^-_i\quad\text{and} $\\
$k+1-G^+\leq e^{-},\,\text{with}\quad G^+=\sum^{e^+}_{j=1}g^+_j.$
\end{center}
Where $g_i^+$'s $(g_{j}^{-})$'s are the weights on the incoming $($outgoing$)$ edges incident to the vertex $v$. 
\end{proposition}
 
\begin{proof}
Suppose $\Lambda$ is the basic set of $\phi_{t}$ corresponding to $v$. First of all, suppose that both edges of the cycle are incoming edges incident to the vertex $v$, as shown in Figure~\ref{singularidade4}.
\begin{figure}[h!]
\centering\includegraphics[scale=0.5]{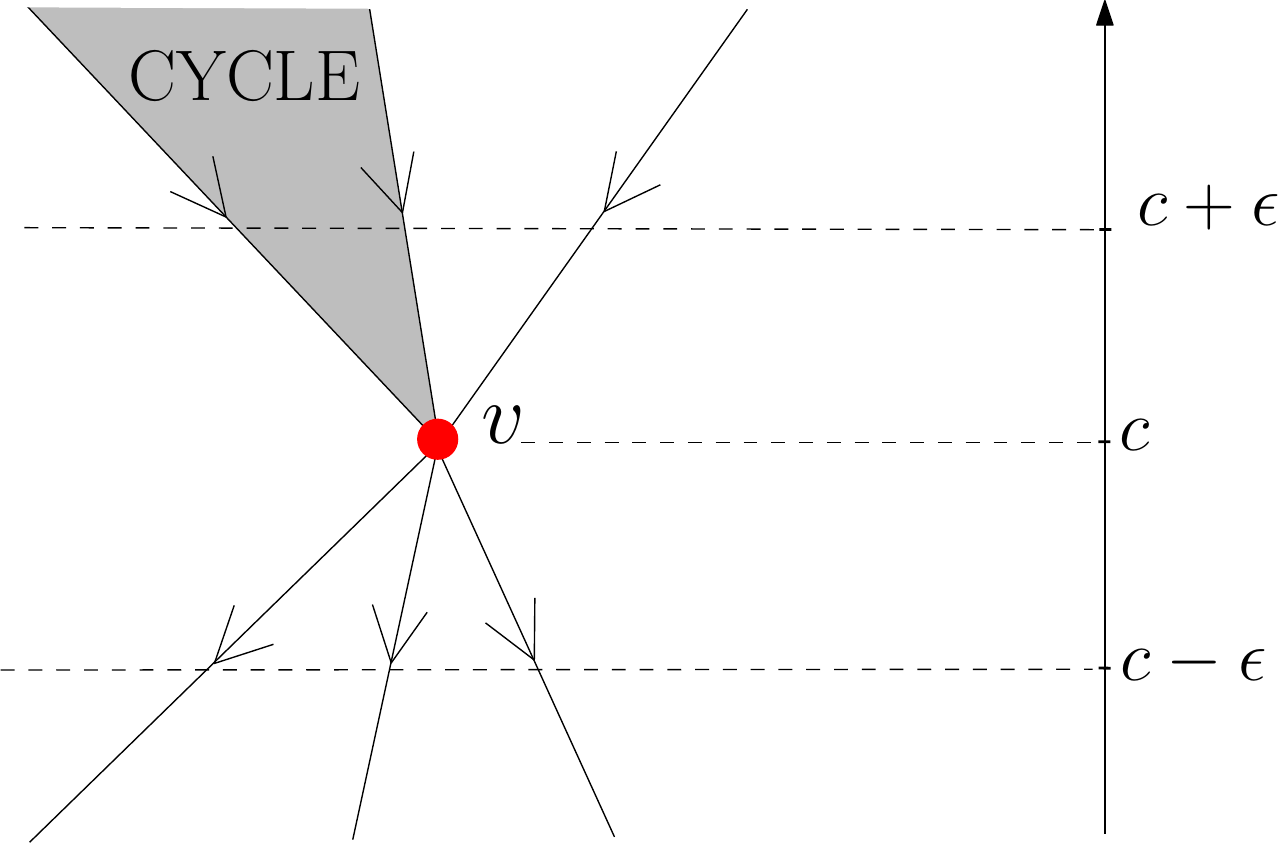}
\caption{$v$ is a vertex in the cycle.}\label{singularidade4}
\end{figure} 
Let $Z$ be the union of the components of $\ss-(\partial U\cap Z_0)$ which do not contain $\Lambda$. By Lemma~\ref{le4}, each component of $Z$ is a manifold of handlebody type. Now we call $X=Z\cup U$ and $Y=\ss-X$. Suppose $Y_1$ is a component of $Y$ such that it corresponds to the cycle on the graph. Therefore, by the Lemma~\ref{le14} one has
\begin{equation}\label{eq29}
\dim H_1(Y_1)=g^+_1+g^+_2\quad\text{and}\quad\dim H_2(Y_1)=1.
\end{equation}
and any other component of $Y$ is of handlebody type.
On the other hand a result (9.11) of~\cite{franks1982} says that $H_2(X,Z)$ has dimension $k$.  
We consider  the exact homology sequence of the pair $(X,Z)$,
\begin{equation}\label{seq1}
H_3(X,Z)\rightarrow H_2(Z)\rightarrow H_2(X)\rightarrow H_2(X,Z)\rightarrow H_1(Z)\stackrel{\alpha}{\rightarrow} H_1(X)\rightarrow \ldots
\end{equation}
Since, $H_3(X,Y)=0$ and $H_2(Z)=0$, this implies that
\begin{equation}\label{eq18}
\dim H_2(X)\leq \dim H_2(X,Z)=k.
\end{equation}
On other hand, we consider the Mayer-Vietoris exact homology sequence,
\begin{equation}
H_3(X)\oplus H_3(Y)\rightarrow H_3(\ss)\rightarrow H_2(X\cap Y)\rightarrow H_2(X)\oplus H_2(Y)\stackrel{\beta}{\rightarrow}H_2(\ss)\rightarrow \ldots
\end{equation}
Both $X$ and $Y$ are compact $3-$manifolds with boundary, so $H_3(X)=0=H_3(Y)$. Also, $X\cap Y$ is composed of $e^+$ closed orientable surfaces, so $H_2(X\cap Y)\cong F^{e^+}_2$. On the other hand, we know that $H_3(\ss)\cong F_2\cong H_2(\ss)$ contains a non separating regular level set. By Lemma~\ref{le16}, it follows that $\beta$ is surjective. Therefore, 
\begin{equation*}
\dim H_2(X)=e^+ -1
\end{equation*} 
By inequality $(\ref{eq18})$, one has 
\begin{equation*}
e^+\leq k+1.
\end{equation*}
If $a=\dim \im\alpha$, from the exact sequence (\ref{seq1}), one has
$$k=e^+ -1+G^- - a.$$ 
Hence,
\begin{equation}\label{eq32}
k+1-G^-\leq e^+.
\end{equation}
On the other hand, $U$ satisfies the Poincar\'e-Hopf condition (\ref{MainTheorem-eq4}) and thus, 
$$e^+ + G^-=e^- + G^+.$$
By inequality (\ref{eq32}), one has
$$k+1-G^+\leq e^-.$$
Now for the last inequality, we need to consider the reverse flow which switches the roles of $e^+$ and $e^-$ and transform $A$ to $A^t$ so $k$ is unchanged.

We call $W=U\cup Y$ and in this case the flow exits $U$ through $Y$. We consider the exact homology sequence of the pair $(W,Y)$, 
\begin{equation}\label{seq2}
\rightarrow  H_3(W,Y)\rightarrow H_2(Y)\rightarrow H_2(W)\rightarrow H_2(W,Y) \rightarrow \ldots
\end{equation}
By Theorem~\ref{teor1}, one has $H_3(W,Y)=0$ and $H_2(Y)\cong F_2$. Therefore, by sequence (\ref{seq2}) 
\begin{equation}\label{eq20}
\dim H_2(W)\leq k+1.
\end{equation}
Now, consider the Mayer-Vietoris exact homology sequence,
\begin{equation}
\rightarrow H_3(W)\oplus H_3(Z)\rightarrow H_3(\ss)\rightarrow H_2(W\cap Z)\rightarrow H_2(W)\oplus H_2(Z)\stackrel{\gamma}{\rightarrow}H_2(\ss) \rightarrow \ldots
\end{equation}
Both $W$ and $Z$ are compact $3$-manifolds with boundary, so $H_3(W)=0=H_3(Z)$. $W\cap Z$ is composed of $e^-$ closed orientable surfaces, hence $H_2(X\cap Y)\cong F^{e^-}_2$. Since the submanifold $W$ contains a non separating regular level set by Lemma~\ref{le16}, $\gamma$ is surjective. Therefore, $\dim H_2(W)=e^-$ and by inequality (\ref{eq20}) one has,
\begin{center}
$e^-\leq k+1.$
\end{center}
The proof is complete in this case. 

Now, suppose that an edge on the cycle is incoming and incident to  $v$ and another edge is outgoing and incident to  $v$, as shown in Figure~\ref{singularidade5}.
\begin{figure}[h!]
\centering\includegraphics[scale=0.5]{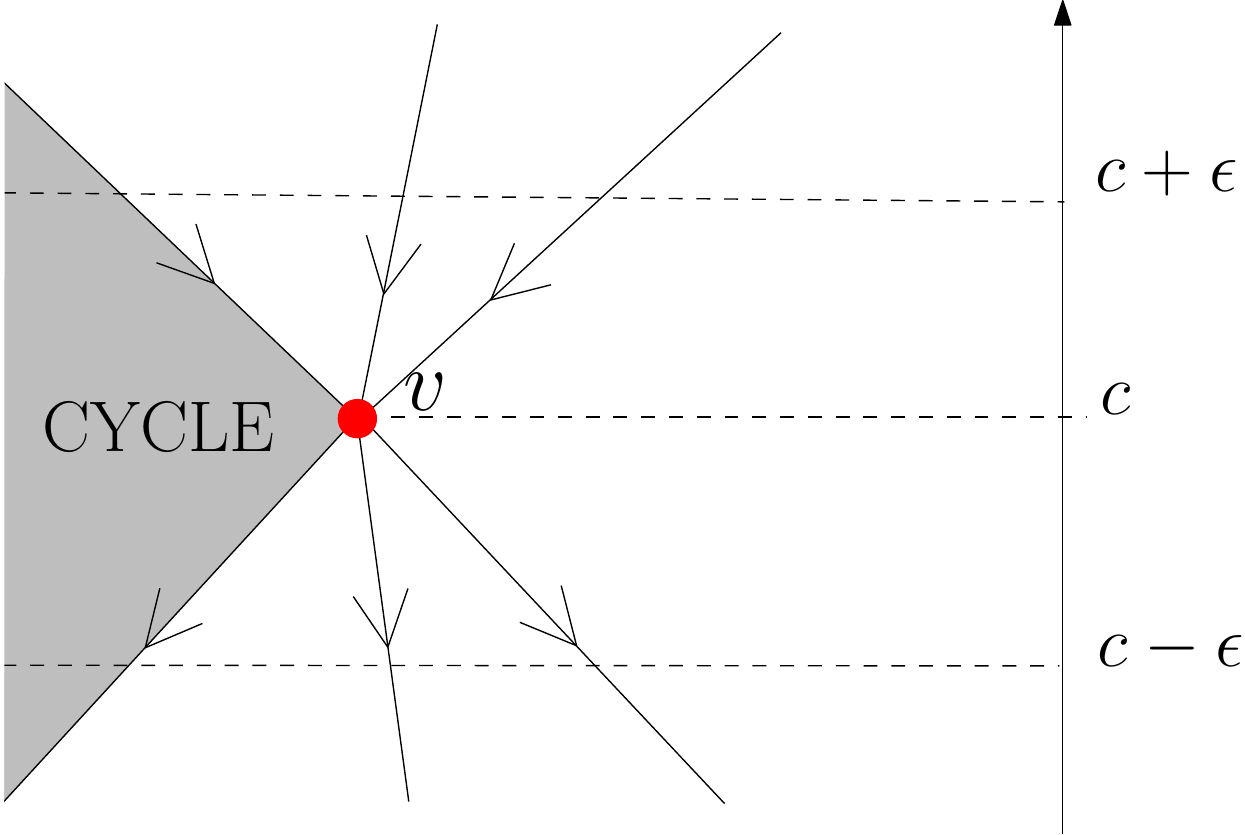}
\caption{$v$ is a vertex in the cycle.}\label{singularidade5}
\end{figure}
We have that $\partial U \cap Z_0$ is a union of surfaces and only one of these is non separating and corresponds to the outgoing edge incident to $v$ on the cycle. We call $Z_1$ the connected component of $Z_0$ that contains the non separating surface  aforementioned.

Let $Z$ be the union of $Z_1$ with the components of $\ss-\left(\partial U \cap (Z_0-Z_1) \right)$ which do not contain $\Lambda$. By Lemma~\ref{le4}, each component of $Z-Z_1$ is a manifold of handlebody type. Now we define $X=Z\cup U$ and $Y=\ss-X$. 

We call $Y_1$ the connected component of $Y$ corresponding to the to the cycle.
By Lemma~\ref{le4} one has that each component of $(Y-Y_1)$ is a manifold of  handlebody type. Note that, $Z\cap Y$ is a closed orientable surface of genus $g$. Moreover, by  Lemma~\ref{le16} one has $H_2(Z)\cong F_2$ and $\di H_1(Z)=\sum g^-_i+g$. With this notation, the result (9.11) of~\cite{franks1982} says that $H_2(X,Z)$ has dimension $k$.

Now, consider the exact homology sequence of the pair $(X,Z)$,
\begin{equation}\label{seq14}
\rightarrow  H_3(X,Z)\rightarrow H_2(Z)\rightarrow H_2(X)\rightarrow H_2(X,Z)\rightarrow H_1(Z)\stackrel{\alpha}{\rightarrow} H_1(X)\rightarrow  \ldots
\end{equation}
Since, $H_3(X,Y)=0$ and $H_2(Z)\cong F_2$, this implies 
\begin{equation}\label{eq28}
\di H_2(X)\leq k+1.
\end{equation}
 Consider the Mayer-Vietoris exact homology sequence,
\begin{equation}\label{secquancia-MV}
0\rightarrow H_3(\ss)\rightarrow H_2(X\cap Y)\rightarrow H_2(X)\oplus H_2(Y)\stackrel{\beta}{\rightarrow}H_2(\ss)\rightarrow \ldots
\end{equation}
Since, $X\cap Y$ is composed of $e^+ +1$ closed orientable surfaces, hence $H_2(X\cap Y)\cong F^{e^++1}_2$. In addition, we know that $H_3(\ss)\cong F_2\cong H_2(\ss)$. Furthermore, there exists a non separating regular level set contained in $Y$. By Lemma~\ref{le16},  $\beta$ is surjective. Therefore, 
\begin{equation*}
\dim H_2(X)=e^+.
\end{equation*} 
By inequality $(\ref{eq28})$, one has 
\begin{equation*}
e^+\leq k+1.
\end{equation*}

On the other hand, as $\beta$ in (\ref{secquancia-MV}) is surjective, one has 
\begin{equation*}
0\rightarrow H_1(X\cap Y)\rightarrow H_1(X)\oplus H_1(Y)\rightarrow H_1(\ss)\rightarrow H_0(X\cap Y)\rightarrow H_0(X)\oplus H_0(Y)\rightarrow H_0(\ss)\rightarrow 0
\end{equation*}
Since $\dim H_1(X\cap Y)=2G^++2g$, $\dim H_1(Y)=G^++g$, $\dim H_0(X\cap Y)=e^++1$ and $\dim H_0(Y)=e^+$, hence 
\begin{equation*}
\dim H_1(X)=G^++g
\end{equation*} 
Since $\partial X\cap\partial Z$ is a genus $g$ surface, then $H_1(X)$ and $H_1(Z)$ both have a $F_2^g$ summand that comes from this surface. Then $\dim \im\alpha\geq g$. 
Also, if $a=\di\im\alpha$, from the exact sequence (\ref{seq14}), one has
$$k=e^+ -1+\sum g^-_i+g - a.$$ 
as $g-a\leq 0$, hence
\begin{equation}\label{eq19}
k+1-G^-\leq e^+.
\end{equation}
Moreover, $U$ satisfies the Poincar\'e-Hopf condition (\ref{MainTheorem-eq4}) and thus, 
$$e^+ + G^-=e^- + G^+.$$
By inequality (\ref{eq19}), one has
$$k+1-G^+\leq e^-.$$
Now for the last inequality, we  consider the reverse flow which switches the roles of $e^+$ and $e^-$ and transforms $A$ to $A^t$ so that $k$ remains unchanged.
\end{proof}

Note that by Lemma~\ref{le4} every other component of the  chain recurrent set that corresponds to a vertex which is not on the cycle is inside a manifold of handlebody type and all separating regular level sets are $\pi_1$-trivial. In this context we have the following proposition.

\begin{proposition}\label{pr5}
Let  $L$ be a Lyapunov graph associated with a Smale flow  $\phi_{t}$ and a Lyapunov function on $\ss$ .
 Let $v$ be a vertex of $L$ labelled with a suspension of a subshift of finite type and $A_{m\times m}$ is the non-negative integer matrix representing this subshift and $k=\di\ke\overline{I-A}:F_2^m\rightarrow F_2^m$ where  $\overline{I-A}$ is the $\mo 2$ reduction of $I-A$. Let $e^+$ and $e^-$ be respectively the indegree and outdegree of  $v$. If $v$ represents a basic set, which is contained in some manifold of handlebody type $M$ in $\ss$, one has:
\begin{center}
$e^+ ,e^- >0,$\\
$e^+\leq k+1,$\\
$e^-\leq k+1,$\\
$k+1-G^-\leq e^{+},\,\text{with}\quad G^-=\sum^{e^-}_{i=1}g^-_i\quad\text{and} $\\
$k+1-G^+\leq e^{-},\,\text{with}\quad G^+=\sum^{e^+}_{j=1}g^+_j,$
\end{center}
where $g_i^+$'s $(g_{j}^{-})$'s are the weights on the incoming $($outgoing$)$ edges incident to the vertex $v$. 
\end{proposition}
\begin{proof}
By Remark~\ref{ob1}, we can suppose that the component $Y_1$ is not a manifold of handlebody type. Then by Corollary~\ref{co1} the other components of $Y$ and all components of $Z$ are manifolds of handlebody type. In other words $H_2(Y)\cong F_2$, $H_2(Z)= 0$ and $\dim H_1(Z)=G^-$.\\
Consider the exact homology sequence of the pair $(X,Z)$,
\begin{equation}\label{seq5}
\rightarrow H_3(X,Z)\rightarrow H_2(Z)\rightarrow H_2(X)\rightarrow H_2(X,Z)\rightarrow H_1(Z)\stackrel{\alpha}{\rightarrow} H_1(X)\rightarrow \ldots
\end{equation}
Since, $H_3(X,Z)= 0$ and $H_2(Z)= 0$, this implies that
\begin{equation}\label{eq17}
\dim H_2(X)\leq \di H_2(X,Z)=k.
\end{equation}
Now consider the exact homology reduced sequence of the pair $(\ss,Y)$,
\begin{equation*}\label{seq13}
\rightarrow \widetilde{H}_1(Y)\rightarrow \widetilde{H}_1(\ss)\rightarrow\widetilde{H}_1(\ss,Y)\rightarrow \widetilde{H}_0(Y)\rightarrow\widetilde{H}_0(\ss)
\end{equation*}
Since, $\widetilde{H}_0(\ss)= 0$ and $\dim(\widetilde{H}_0(Y))=e^+-1$, we conclude that 
\begin{center}    
  $\dim \widetilde{H}_1(\ss,Y)\geq e^+-1.$
\end{center} 
Moreover, by Lefschetz duality $H_{1}(\ss,Y)\cong H_{2}(\ss-Y,\ss-\ss)\cong H_{2}(X)$. Therefore, 
\begin{equation}\label{eq30}
\di H_{2}(X)\geq e^+-1
\end{equation}
and from inequality $(\ref{eq17})$, one has
\begin{equation*}
e^+\leq k+1
\end{equation*} 
Now we consider the Mayer-Vietoris exact homology sequence,
\begin{equation*}
0\rightarrow H_3(\ss)\rightarrow H_2(X\cap Y)\rightarrow H_2(X)\oplus H_2(Y)\rightarrow H_2(\ss)\rightarrow \ldots
\end{equation*}
Since, $X\cap Y$ is composed of $e^+$ closed orientable surfaces, one has that $H_2(X\cap Y)\cong F^{e^+}_2$. Moreover, we know that $H_3(\ss)\cong F_2\cong H_2(\ss)$. Therefore,
\begin{equation}\label{eq25}
\dim H_2(X)\leq e^+ -1.
\end{equation}
By inequalities $(\ref{eq30})$ and $(\ref{eq25})$, one concludes that
\begin{equation*}
\dim H_2(X)= e^+ -1.
\end{equation*}
If $a=\dim \im\alpha$, from exact sequence (\ref{seq5}), one has that
\begin{equation*}
k=\dim H_2(X)+G^- - a.
\end{equation*}
Hence,
\begin{equation}\label{eq26}
k+1-G^-\leq e^+.
\end{equation} 
On the other hand, $U$ satisfies the Poincar\'e-Hopf condition (\ref{MainTheorem-eq4})  and thus, 
$$e^+ + G^-=e^- + G^+.$$
Thus,
$$k+1-G^+\leq e^-.$$
Now for the last inequality, we need to consider the reverse flow which switches the roles of $e^+$ and $e^-$ and transform $A$ to $A^t$ so that $k$ is unchanged. In this case, the flow enters $U$ through $Z$ and exits through $Y$. We call $W=Y\cup U$ and we consider the exact homology sequence of the pair $(W,Y)$,  
\begin{equation}\label{seq6}
\rightarrow H_3(W,Y)\rightarrow H_2(Y)\rightarrow H_2(W)\rightarrow H_2(W,Y)\rightarrow\dots.
\end{equation}
By Theorem~\ref{teor1}, one has $H_3(W,Y)= 0$ and $H_2(Y)\cong F_2$. Therefore, by  sequence (\ref{seq6}) 
\begin{equation}\label{eq27}
\dim H_2(W)\leq k+1\quad\text{and}\quad\dim H_2(W)\geq 1.
\end{equation}

Now consider the exact sequence of the pair $(\ss,Z)$ and by Lemma~\ref{le13} one has $\dim H_2(W)=e^-$. It follows from (\ref{eq27}) that
\begin{center}
$e^-\leq k+1.$
\end{center}
\end{proof}
\subsection{Lyapunov Graph without Cycles}\label{without-cycle}

We continue to use the notation established in the beginning of this section. Let $U$ be the closure of the component of $X_0-Z_0$ which contains $\Lambda$. Then $U$ is a neighborhood of $\Lambda$ whose boundary consists of closed orientable surfaces. Also, the flow is transverse to the boundary of $U$. The flow enters  $e^+$ of these surfaces and exits the remaining $e^-$ surfaces. Let $Z$ be the union of the components of $\ss-(\partial U\cap Z_0)$ which do not contain $\Lambda$.
Now define $X=Z\cup U$ and $Y=\ss-X$. With this notation, the result (9.11) of~\cite{franks1982} says that $H_2(X,Z)$ has  dimension $k$.

First of all, suppose that some component of $\partial U^-$ is a $\pi_1$-non-trivial regular level set. Furthermore, we can suppose that $\partial Z_1$ is also a $\pi_1$-non-trivial regular level set. By Lemma~\ref{le4} and Corollary~\ref{co1}, one has that each component of $Z$ and each component of $Y$ are manifolds of handlebody type. 

\begin{proposition}\label{pr7}
Let  $L$ be a Lyapunov graph associated with a Smale flow  $\phi_{t}$  and a Lyapunov function on $\ss$, such that $L$ is a tree. Furthermore, suppose there exists a $\pi_1$-non-trivial regular level set. Let $v$ be a vertex of $L$ labelled with a suspension of a subshift of finite type and $A_{m\times m}$ is the non-negative integer matrix representing this subshift. Let $k=\di\ke\overline{I-A}:F_2^m\rightarrow F_2^m$ where  
$\overline{I-A}$ is the $\mo 2$ reduction of $I-A$. Then, if $e^+$ and $e^-$ are respectively the indegree and outdegree of $v$, one has:
\begin{center}
$e^+ ,e^- >0,$\\
$e^+\leq k+1,$\\
$e^-\leq k+1,$\\
$k+1-G^-\leq e^{+},\,\text{with}\quad G^-=\sum^{e^-}_{i=1}g^-_i\quad\text{and} $\\
$k+1-G^+\leq e^{-},\,\text{with}\quad G^+=\sum^{e^+}_{j=1}g^+_j,$
\end{center}
where $g_i^+$'s $(g_{j}^{-})$'s are the weights on the incoming $($outgoing$)$ edges incident to the vertex $v$. 
\end{proposition}

\begin{proof}
Suppose $\Lambda$ is the basic set of $\phi_{t}$ corresponding to $v$. 
Consider  the exact homology sequence of the pair $(X,Z)$,
\begin{equation}\label{seq3}
\rightarrow H_3(X,Z)\rightarrow H_2(Z)\rightarrow H_2(X)\rightarrow H_2(X,Z)\rightarrow H_1(Z)\stackrel{\alpha}{\rightarrow} H_1(X)\rightarrow \ldots
\end{equation}
Since, $H_3(X,Y)=0$ and $H_2(Z)=0$, this implies that
\begin{equation}\label{eq21}
\dim H_2(X)\leq k.
\end{equation}
Now consider the exact homology reduced sequence of the pair $(\ss,Y)$,
\begin{equation}\label{seq4}
\rightarrow \widetilde{H}_1(Y)\rightarrow \widetilde{H}_1(\ss)\rightarrow\widetilde{H}_1(\ss,Y)\rightarrow \widetilde{H}_0(Y)\rightarrow\widetilde{H}_0(\ss)\rightarrow 0
\end{equation}
Since, $\widetilde{H}_0(\ss)=0$ and $\dim(\widetilde{H}_0(Y))=e^+-1$, we conclude that 
\begin{center}    
  $\dim \widetilde{H}_1(\ss,Y)\geq e^+-1.$
\end{center} 
Moreover, by Lefschetz duality $H_{1}(\ss,Y)\cong H_{2}(\ss-Y,\ss-\ss)\cong H_{2}(X)$. Therefore, $\di H_{2}(X)\geq e^+-1$ and from inequality $(\ref{eq21})$, one has
\begin{equation*}
e^+\leq k+1
\end{equation*} 
Also, from sequence (\ref{seq4}), we have that 
\begin{equation}\label{eq23}
\dim H_2(X)\leq e^+-1+\dim H_1(\ss)=e^+.
\end{equation}
If $a=\di \im\alpha$, from the exact sequence (\ref{seq3}), one has
\begin{equation}\label{eq22}
k=\dim H_2(X)+G^- - a.
\end{equation}
 
The hypothesis that there exits a $\pi_1$-non-trivial regular level set in $Z$, implies that $\im\alpha\neq 0$. Therefore, from inequality (\ref{eq22}) we can conclude that
\begin{equation}
k<\dim H_2(X)+G^-.
\end{equation}
Now by inequality (\ref{eq23}), one has that
\begin{equation*}
k+1-G^-\leq e^+.
\end{equation*}
On the other hand, $U$ satisfies the Poincar\'e-Hopf condition (\ref{MainTheorem-eq4})  and thus, 
\begin{center}
$e^+ + G^-=e^- + G^+.$
\end{center}
Hence, one has
$$k+1-G^+\leq e^-.$$
Now for the last inequality, we need to consider the reverse flow which switches the roles of $e^+$ and $e^-$ and transforms $A$ to $A^t$ so that $k$ is unchanged.

For the general case, we can suppose that a $\pi_1$-non-trivial regular level set is contained in some component of $Z$. Then it suffices to repeat the arguments in the above proof.
\end{proof}
Now we suppose that each connected component of a regular level set is $\pi_1$-trivial in $\ss$.

 Let $k=\di \ke\overline{I-A}:F^m_2\rightarrow F^m_2$ where  $\overline{I-A}$ is the $\mo 2$ reduction of $I-A$ and $e^+$ ($e^-$) is the indegree (outdegree) of $v$. With the definition of $X$, $Y$ and $Z$ as defined in the beginning of Subsection~\ref{without-cycle}, one has the following lemma.

\begin{proposition}\label{le7}
Let $v$ be a vertex of $L$ labelled with a suspension of a subshift of finite type and $A_{m\times m}$ is the non-negative integer matrix representing this subshift. Suppose each component of $Y$ and $Z$ are manifolds of handlebody type then one has:
\begin{center}
$e^+ ,e^- >0,$\\
$e^+\leq k,$\\
$e^-\leq k,$\\
$k-G^-\leq e^{+},\,\text{with}\quad G^-=\sum^{e^-}_{i=1}g^-_i\quad\text{and} $\\
$k-G^+\leq e^{-},\,\text{with}\quad G^+=\sum^{e^+}_{j=1}g^+_.$
\end{center}
where $g_i^+$'s $(g_{j}^{-})$'s are the weights on the incoming $($outgoing$)$ edges incident to $v$.
\end{proposition}

\begin{proof}
Since, $Z$ and $Y$ are disjoint unions of manifolds of handlebody type, one has that $H_2(Z)=0=H_2(Y)$, $\dim H_1(Z)=G^-$ and $\dim H_1(Y)=G^+$. Now  consider the exact homology sequence of the pair $(X,Z)$,
\begin{equation}\label{seq7}
\rightarrow H_3(X,Z)\rightarrow H_2(Z)\rightarrow H_2(X)\rightarrow H_2(X,Z)\rightarrow H_1(Z)\stackrel{\alpha}{\rightarrow} H_1(X)\rightarrow \ldots
\end{equation}
Since $H_3(X,Z)=0$ and $H_2(Z)=0$, one has
\begin{equation}\label{eq24}
\di H_2(X)\leq k. 
\end{equation}
Consider the reduced exact homology sequence of the pair $(\ss,Y)$,
\begin{equation}\label{seq8}
\rightarrow H_1(Y)\stackrel{\beta}\rightarrow H_1(\ss)\rightarrow H_1(\ss,Y)\rightarrow \widetilde{H}_0(Y)\rightarrow 0.
\end{equation}
By Lemma~\ref{le13} one has that $\beta$ is trivial and $\di \widetilde{H}_0(Y)=e^-+1$. Thus, by sequence $(\ref{seq8})$, one has  
      $$\dim \widetilde{H}_1(\ss,Y)= e^+.$$ 
Using  Lefschetz duality $H_{1}(\ss,Y)\cong H_{2}(\ss-Y,\ss-\ss)\cong H_{2}(X)$. Therefore, $\di H_{2}(X)= e^+$ and from equality $(\ref{eq24})$, one has
\begin{equation*}
e^+\leq k.
\end{equation*}
 On the other hand, if $a=\di\im \alpha$ one has from  sequence~(\ref{seq7}),
 $$k=\di H_2(X)+\di H_1(Z)-a.$$
Therefore,
$$k-G^-\leq e^+$$ 
Furthermore, $U$ satisfies the Poincar\'e-Hopf condition (\ref{MainTheorem-eq4}) and thus,
$$e^+ + G^-=e^- + G^+.$$
Hence, one has
$$k-G^+\leq e^-.$$
Now for the last inequality, we need to consider the reverse flow which switches the roles of $e^+$ and $e^-$ and transforms $A$ to $A^t$ so that $k$ is unchanged.
\end{proof}
 
\textbf{Proof of the necessity of the conditions of Theorem~\ref{MainTheorem}.}
\begin{enumerate}
\item items (1) and (2) follow from Corollary 3.1 in~\cite{de1993gradient};
\item item (4) follows by Theorem 4.7 in \cite{cruz1998cycle};
\item item (3) follows from Proposition~\ref{co};
\item item 3(a) follows from Propositions~\ref{pr6} and~\ref{pr5};
\item item 3(b)(i) follows from Propositions~\ref{pr5},~\ref{pr7} and~\ref{le7};
\item item 3(b)(ii),\\
if there exists a regular level set $\mathcal{T}_g$ which is $\pi_1$-non-trivial, then clearly $g\neq 0$. Otherwise, if each regular level set is $\pi_1$-trivial, by the use of Lemma~\ref{le4} on the edges of the graph we obtain that there exists a  vertex $v_0$ labelled with a suspension of a subshift of finite type associated to the non-negative integer matrix  $A_{m\times m}$ such that
\begin{center}
$e^+ ,e^- >0,$\\
$e^+\leq k,$\\
$e^-\leq k,$\\
$k-G^-\leq e^{+},\,\text{with}\quad G^-=\sum^{e^-}_{i=1}g^-_i\quad\text{and} $\\
$k-G^+\leq e^{-},\,\text{with}\quad G^+=\sum^{e^+}_{j=1}g^+_j,$
\end{center}
where $k=\di\ke\left(\left(I-B\right):F^{m}_{2}\rightarrow F^{m}_{2}\right)$, $F_{2}=\Z/2$, $b_{ij}=a_{ij}\mo 2$ and $g_{i}^{+}(g_{j}^{-})$ are the weights on the incoming $($outgoing$)$ edges incident to the vertex $v_0$. By hypothesis,
$$
k-G^-<k+1-G^-\leq e^-\leq k 
$$
so $k-G^-<k$, that is $G^->0$ thus there is at least one $g^-_j>0$. Similarly, there is at least one $g^+_k>0$.
\end{enumerate}
\cqd

\section{Sufficient Conditions on Abstract Lyapunov Graphs associated to Smale Flows}\label{sec-suff}

In this section we prove propositions that will combine in order to form the proof of the sufficiency of the conditions stated in the main theorem of this work, Theorem~\ref{MainTheorem}.
In $\cite{de1987smale}$ the construction of basic block for singularities, periodic orbits and subshifts of finite type that verify the conditions in Theorem~\ref{teorem1} are presented. Hence, there is no need to present thes herein. In what follows, we present a construction of a basic block for a vertex $v$  that satisfies the condition of Proposition~\ref{le7}. For that, we need some definitions. 
\begin{definition} Let $C_1$ and $C_2$ be solid concentric cylinders with $C_2\subset C_1$. A \textbf{round handle $R$} is a $3-$manifold homeomorphic to $C_1-C_2$ containing a saddle type periodic orbit of period equal to one and with a flow defined on $R$ as follows: the flow enters  two disjoint boundary components of $R$ homeomorphic to annuli and exits two other disjoint boundary components also homeomorphic to annuli. See Figure~\ref{figure9}
\end{definition}
\begin{figure}[h!]
  \centering{\includegraphics[width=5cm]{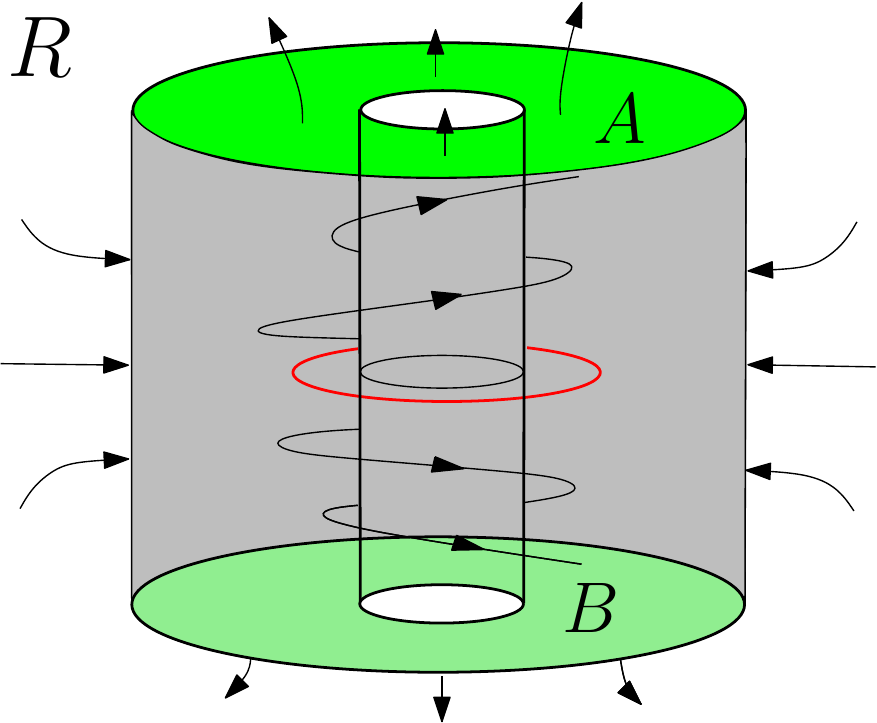}}
\caption{Round handle.}\label{figure9}
\end{figure}
\begin{definition}
A one-handle $H_i$ to a strip $D^1\times D^1$ and is transversal to the flow. A one-handle $H_s$ in a round handle $R_s$ will be chosen so that $H_s\times \mathbb{S}^1=R_s$. A one-handle $H_t$ in a nilpotent handle $N_t$ will be chosen so that $H_t\times D^1=N_t$.
\end{definition}
\begin{definition}
A \textbf{basic block} for a one-dimensional set $\Lambda$ of a Smale flow $\phi_t$ on $\ss$ and Lyapunov function $f:\ss\rightarrow \re$ with $f(\Lambda)=c$ is the component of $X=f^{-1}([c-\epsilon,c-\epsilon])$ which contain $\Lambda$, where $\epsilon>0$ is chosen so that $X$ contains on other basic set and
\begin{enumerate}
\item there exists $($ not necessarily connected $)$ codimension one submanifolds with boundary $U$ and $V$ in $X$ with $U\subset V$ is everyehere transversal to the flow.
\item The first return map $\mu:U\rightarrow \inte V$ is a well defined smooth map and there is a hyperbolic handle set $H\subset U$ with every orbit of $\Lambda$ intersecting $H$ and every $h_i\subset H$ intersecting $\Lambda$.
\item if $x\in H$ but $r(x)\notin H$ then $\phi_t(x)\cap H=\emptyset$ $\forall t>0$ and $f(\phi_{t_0}(x))=c-\epsilon$ for some $t_0>0$. Likewise, if $x\in H$ but $r^{-1}(x)\notin H$ then $\phi_t(x)\cap H=\emptyset$ $\forall t<0$ and $f(\phi_{t_0}(x))=c+\epsilon$ for some $t_0<0$.
\end{enumerate}
\end{definition}

In the following construction, we build specific Smale flows on handlebodies which are transverse to its boundary. This construction will be very important for the next subsection.
There exists a Smale flow on a handlebody $H_g$ of genus $g$, with $g$ attracting periodic orbits $\lambda_{a_i}$, $g-1$ saddle periodic orbits $\lambda_{s_i}$ and $g-1$ repelling singularities $p_i$. $H_g$  is obtained by gluing $g-1$ round handles to $g$ solid tori as shown in Figure~\ref{figure23}.
\begin{figure}[h!]
  \centering{\includegraphics[width=12cm]{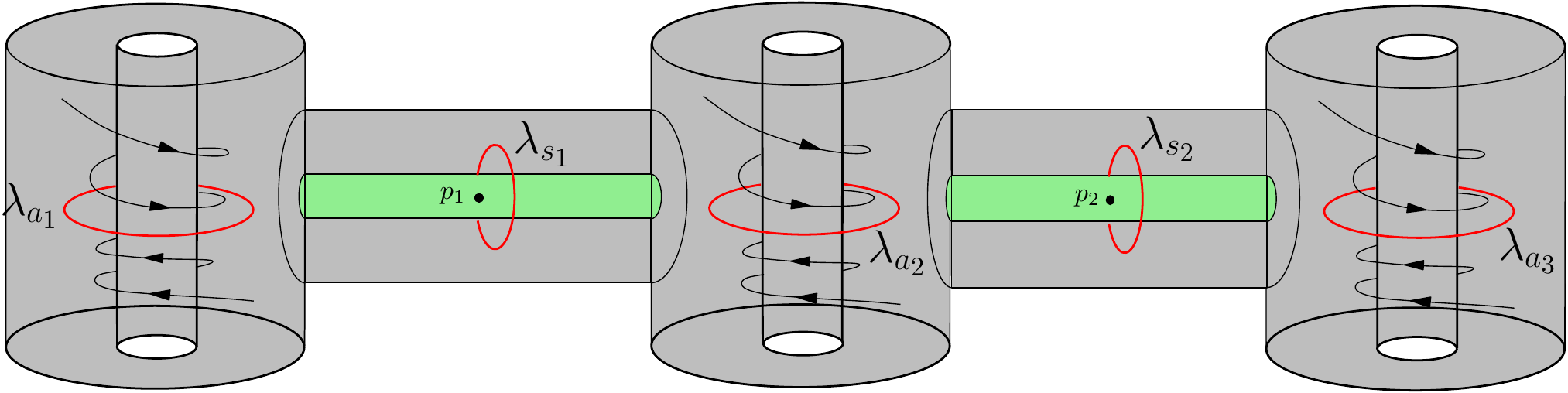}}
\caption{$H_3$ is a handlebody of genus $3$.}\label{figure23}
\end{figure}
The flow on $H_g$ is induced by the periodic orbits and singularities and will be denoted by $\rho_t$. Since the flow is transversal and points inward on  $\partial H_g$, one has 
$$
\epsilon=d(\partial H_g, \rho_1(\partial H_g))>0.
$$
\begin{proposition}\label{pr10}
Let $S$ be a surface homeomorphic to $D^1\times D^1$ embedded in a tubular $\epsilon/2$-neighborhood of the boundary of handlebody, such that the flow described above is transversal to $S$. Then there exists a neighborhood $V$ of $S$ such that the first return map, $\mu:S\rightarrow V$, is smooth. 
\end{proposition}
\begin{proof}
Suppose that $S$ is embedded in some round handle $R_j$ with periodic orbit 
$\lambda_{s_j}$. Hence, one can extend $S$ to a surface $V_S$, which is homeomorphic to $D^1\times D^1$ and has transversal intersection with $\lambda_{s_j}$, as shown in Figure~\ref{figure24}. 
\begin{figure}[h!]
  \centering{\includegraphics[width=13cm]{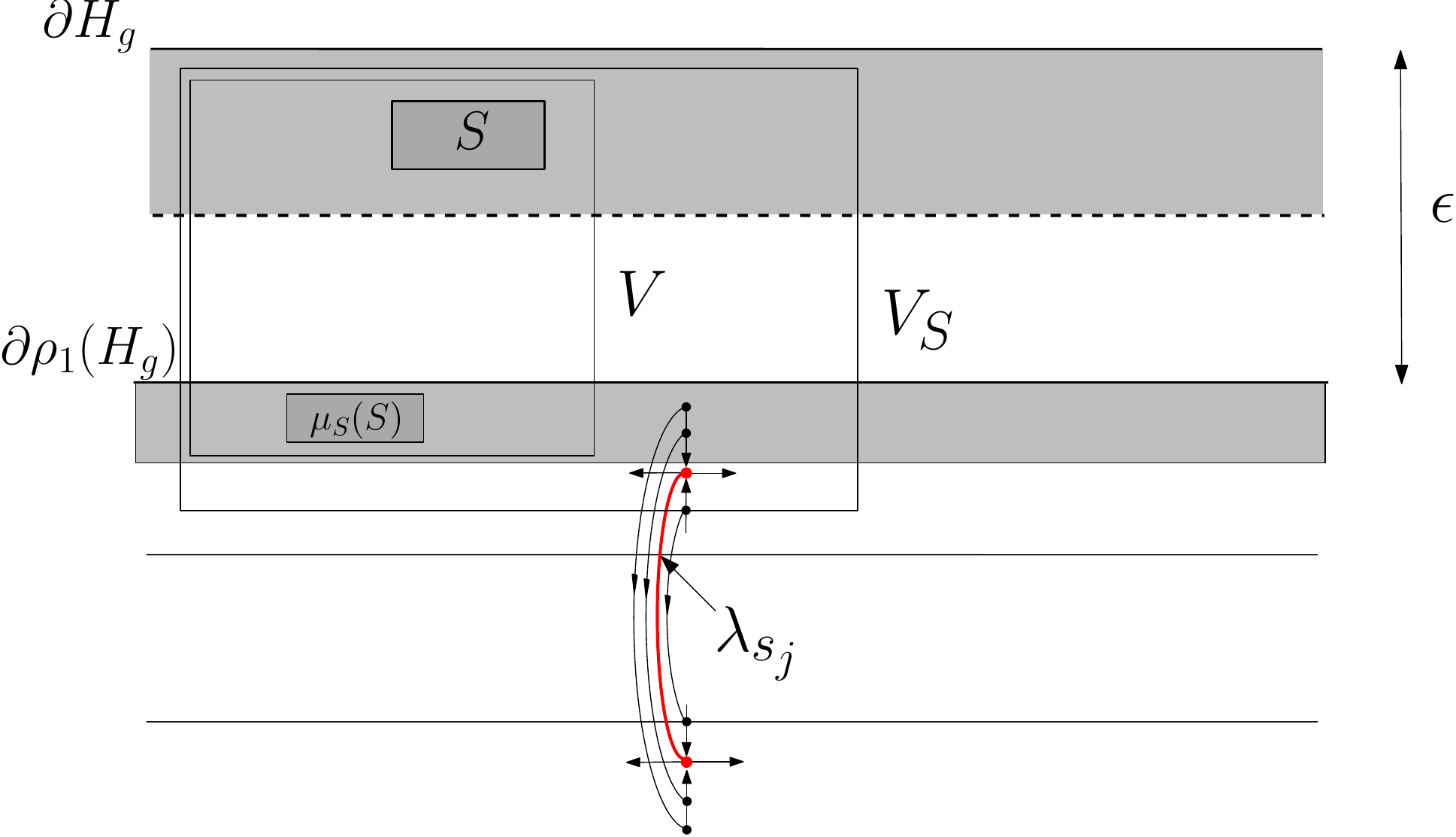}}
\caption{The rectangle $S$ is the interior of $\partial H_g\times I$.}\label{figure24}
\end{figure}
\\
Therefore, $V_S$ is a cross section of $\lambda_{s_j}$, so the first return map $\mu_S$ is defined for $V_S$ and is smooth. Now if it is necessary one can extend $V_S$, such that $V_S$ contains $\mu_S(S)$. On the other hand, since $\lambda_{s_j}$ has an index $1$ periodic orbit, one has that $\mu_S(S)$ does not intersect $\lambda_{s_j}$. Hence, one can find a surface $V\subset V_S$ such that $V$ contains $S$ and $\mu_S(S)$ and does not intersect with $\lambda_{s_j}$. See Figure~\ref{figure24}. Now, suppose that $S$ is contained in some neighborhood of a attracting periodic orbit $\lambda_{a_i}$, one proceeds in a similar way. For the general case, one can choose a smaller neighborhood contained in $S$ and argue as above.
\end{proof} 
Given a handlebody $M$, we consider a flow on $M$ as described above, and  the restriction of this flow to the tubular neighborhood of $\partial M$. Note that, the tubular neighborhood of $\partial M$ is homeomorphic to the collaring of the boundary, $\partial M \times I$, where $I$ is a closed interval. For this reason we refer to the tubular neighborhood of $\partial M$ as the collaring of $\partial M$.
 In $\cite{de1987smale}$ there is a description of distinct ways of gluing a round handle $R$ to a collaring of $\partial M$.

In \cite{de1987smale} three types of gluing of round handles, to a collaring of $\partial M$  are presented. See Figure~\ref{figure1}. In the next subsection we construct a special round handle that can not be realized in $\mathbb{S}^3$. Subsequently, we make use of this special round handle to construct special basic blocks which consequently are also not realizable in $\mathbb{S}^3$.
\begin{figure}[h!]
  \centering{\includegraphics[width=14cm]{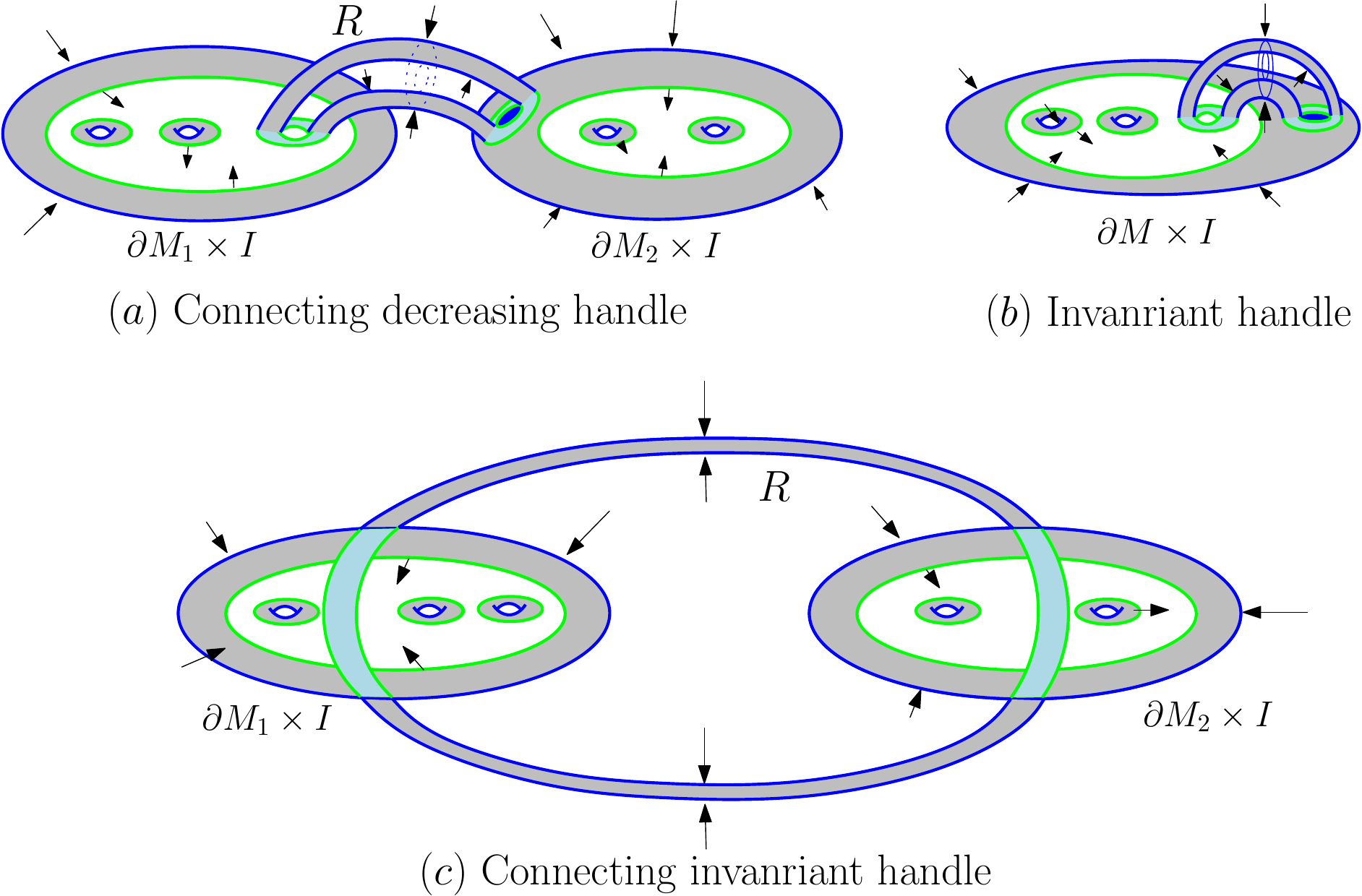}}
\caption{Distinct gluings of round handles.}\label{figure1}
\end{figure}

\subsection{Special Round Handle in $\ss$}

Due to the topology of $\ss$ there is another case to consider. In order to describe it, we build a special round handle $R_1$ in $\ss$. Start with a non separating sphere $\mathbb{S}^2$ and remove two disjoint disk $D_1$ and $D_2$. Now consider a product of $\mathbb{S}^2-(D_1\sqcup D_2)$ with an interval $I$. This manifold is a $3$-manifold, $R_1$, such that the boundary $\partial R_1$ is composed by four annuli. One can put a saddle periodic orbit inside of $R_1$, as shown in Figure~\ref{figureSRH}.
\begin{figure}[h!]
\centering{\includegraphics[width=6cm]{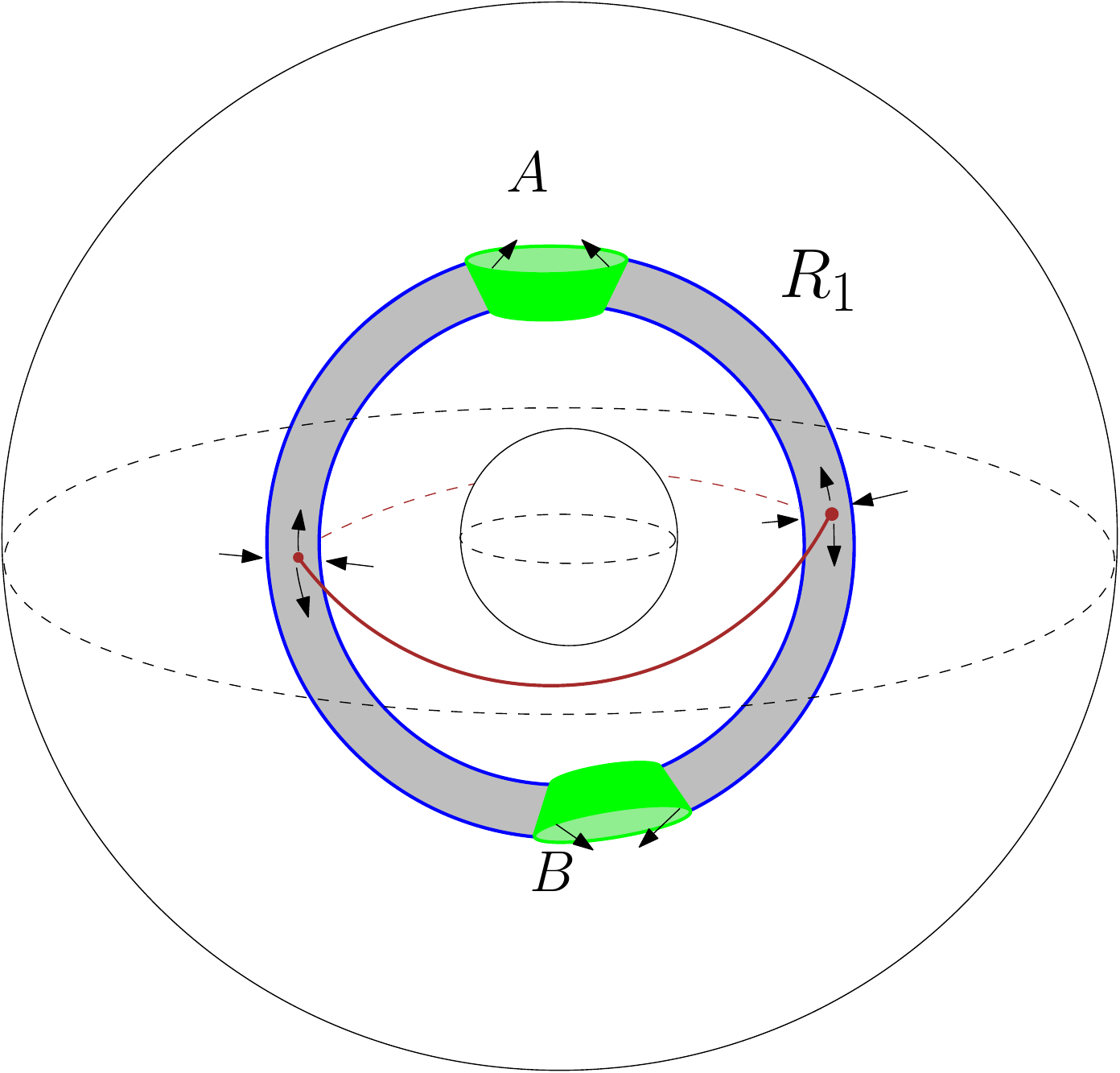}}
\caption{Special type of round handle in $\ss$.}\label{figureSRH}
\end{figure}

Let $N$ be a compact, connected  $3$-manifold with nonempty and connected boundary which is inside a $3$-ball in $\ss$. In the boundary of $N$, $\partial N$ consider two disjoint disks $D_1$ and $D_2$ and inside each one of them other smaller disks $\overline{D}_1\subset D_1$ and $\overline{D}_2\subset D_2$ such that the annuli  $A_1=\overline{D_1-\overline{D}_1}$ and $B_1=\overline{D_2-\overline{D}_2}$ have boundaries $\alpha_1'=\partial D_1$, $\alpha_2'=\partial\overline{D}_1$, $\beta_1'=\overline{D}_2$ and $\beta_2'=\partial \overline{D}_2$. A round handle $R$ is homeomorphic to $\mathbb{S}^1\times [a,b]\times [c,d]$. Consider the following circles $\beta_1=\mathbb{S}^1\times \left\{a\right\}\times \left\{c\right\}$, $\beta_2=\mathbb{S}^1\times \left\{b\right\}\times \left\{c\right\}$,
$\alpha_1=\mathbb{S}^1\times \left\{a\right\}\times \left\{d\right\}$ and $\alpha_2=\mathbb{S}^1\times \left\{b\right\}\times \left\{d\right\}$. Also consider the annuli $A=\mathbb{S}^1\times[a,b]\times \left\{d\right\}$, $B=\mathbb{S}^1\times[a,b]\times \left\{c\right\}$, $E=\mathbb{S}^1\times \left\{a\right\}\times [c,d]$ and $F=\mathbb{S}^1\times \left\{b\right\}\times [c,d]$. With this decomposition we now describe the special gluing of  $R_1$  to the $3$-manifold $N$. The annulus $A$ will be glued to $A_1$ and $B$ to $B_1$ in such a way that $\alpha_1$ is identified to $\alpha_2'$ and $\alpha_2$ is identified to $\alpha_1'$. Also, $\beta_1$ is identified to $\beta_1'$ and $\beta_2$ is identified to $\beta_2'$. See Figure~\ref{figure3}.
\begin{figure}[h!]
\centering{\includegraphics[width=9cm]{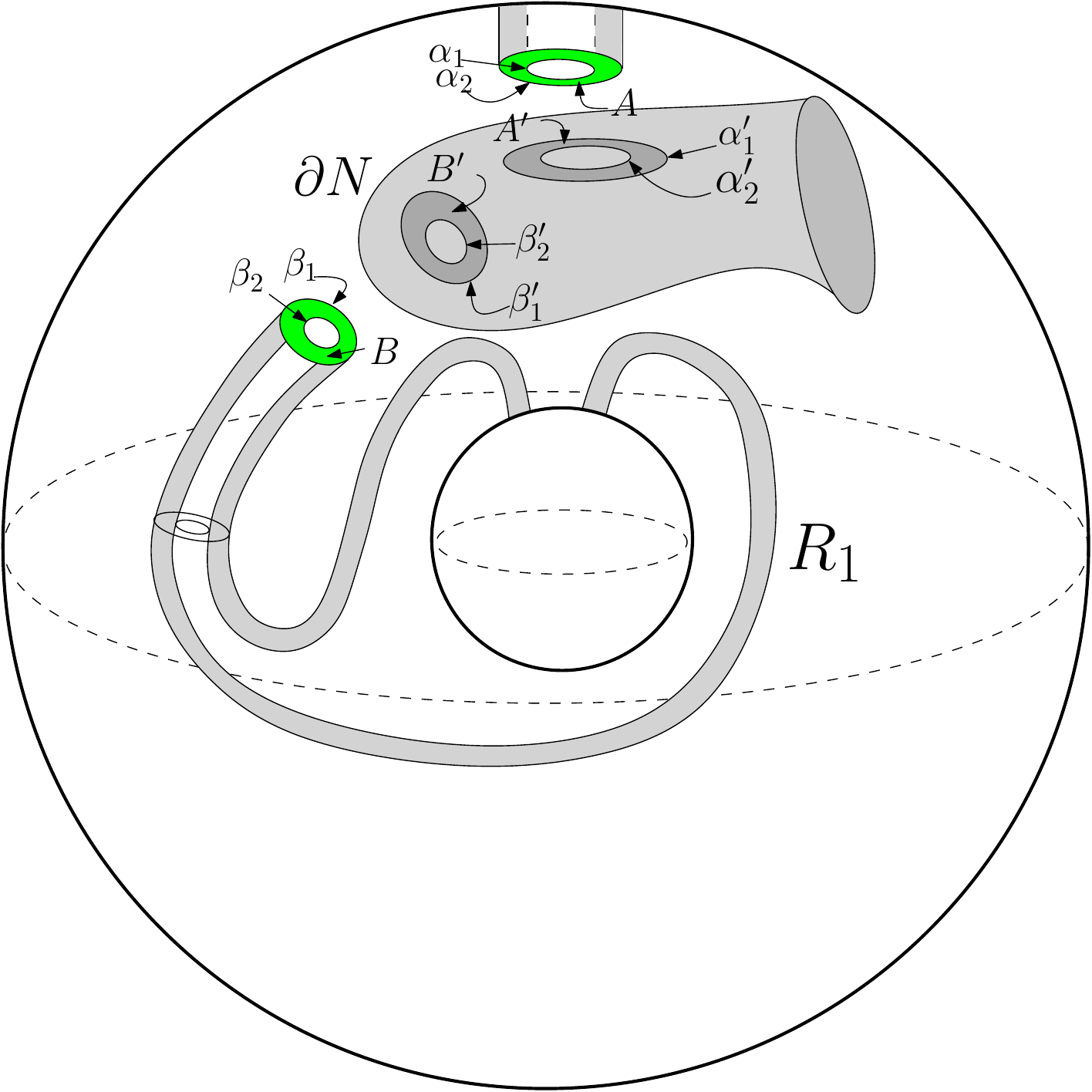}}
\caption{ Special gluing of a round handle in $\ss$.}\label{figure3}
\end{figure}

\begin{lemma}\label{le18}
With the notation above, $\partial(N\cup R_1)$ is homeomorphic to $\partial N$.
\end{lemma}
\begin{proof}
We have that
$$
\partial N= \overline{D}_2\cup_{\beta_2'}B_1\cup_{\beta_1'}C\cup_{\alpha_1'}A_1\cup_{\alpha_2'}\overline{D}_1
$$
where $C=\overline{\partial N-(D_1\cup D_2)}$. On the other hand, one has 
$$
\partial(N\cup R_1)=\overline{D}_1\cup_{\alpha_1=\alpha_1'}E\cup_{\beta_1=\beta_1'}C\cup_{\alpha_2=\alpha_2'}F\cup_{\beta_2=\beta_2'}\overline{D}_2
$$
Note that both $\overline{D}_1\cup E$ and $F\cup \overline{D}_2$ are disks. Hence, $\partial N$ is homeomorphic to $\partial (N\cup R_1)$.
\end{proof}
\begin{corollary}\label{co3}
Let $N$ be a handlebody of genus $g$ embedded in a $3$-ball in $\ss$. Let 
$$
X=N\cup R_1
$$
as described above, then $\ss-X$ is homeomorphic to a handlebody of genus $g$.
\end{corollary}
\begin{proof}
Recalling the construction of $R_1$, we have that
$$
R_1=\overline{\mathbb{S}^2-(D^2_1\cup D^2_2)}\times I
$$
where $\mathbb{S}^2$ is a non separating sphere. The annuli $\partial D^2_1\times I$ and $\partial D^2_2\times I$ are the gluing regions. For the case of a genus zero handlebody, a $3$-disk should be glued to the above annuli using two $2$-disks on the boundary of the $3$-disk, as shown in the Figure~\ref{figure2}.

\begin{figure}[h!]
\centering{\includegraphics[width=7cm]{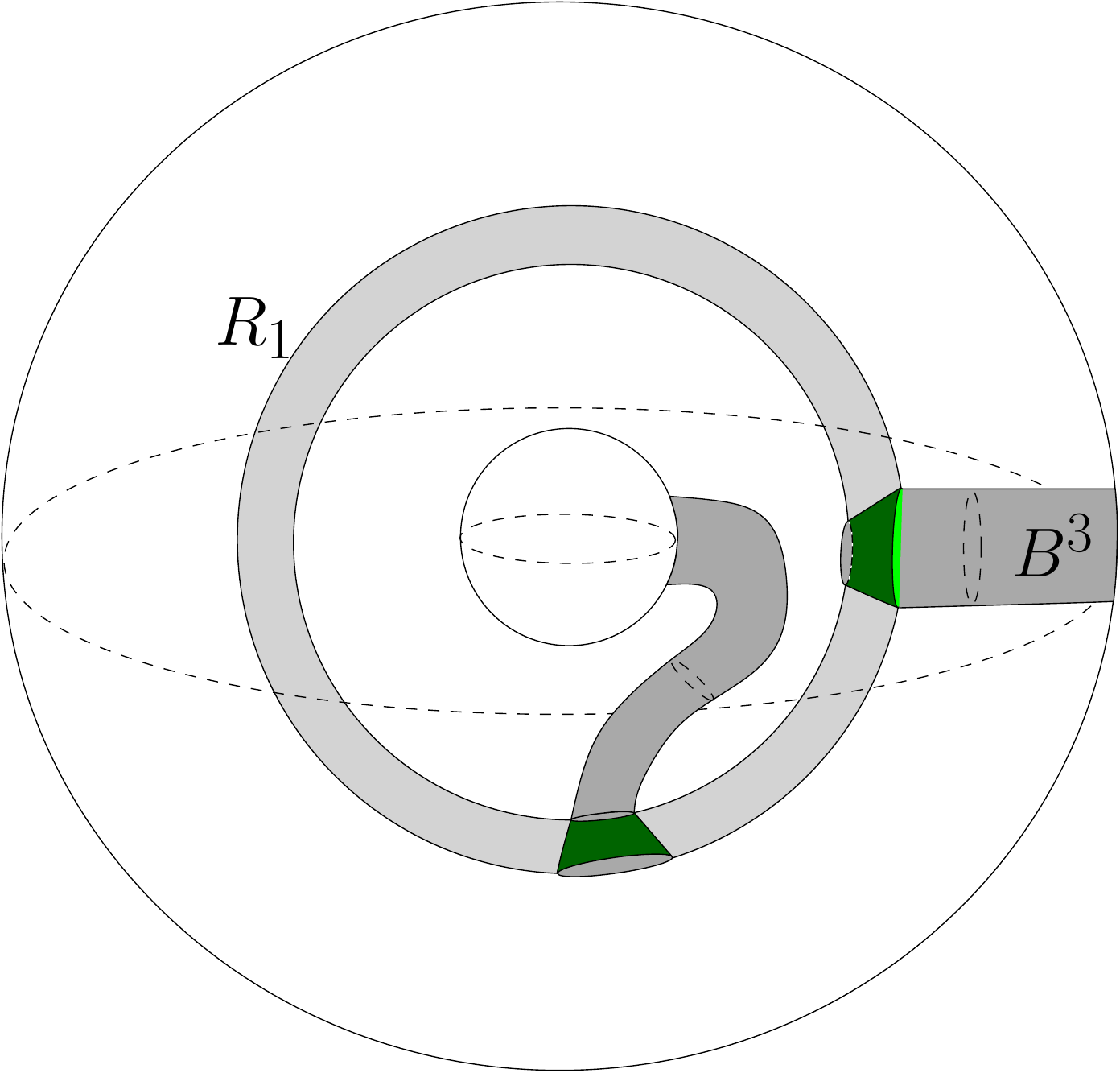}}
\caption{The special round handle $R_1$ glues to $B^3$.}\label{figure2}
\end{figure}
Note that the final manifold is isotopic to a tubular neighborhood of $\mathbb{S}^2_a\vee \mathbb{S}^1_b$, where $\mathbb{S}^2_a$ is a fiber of $\ss\rightarrow \mathbb{S}^1$ and $\mathbb{S}^1_b$ is a fiber of $\ss\rightarrow \mathbb{S}^2$. Thus, the complement in $\ss$ is homeomorphic to $B^3$.\\
In the general case of a handlebody of any genus, it is enough to observe that 
when we glue a handle in $B^3$, we increase the genus of the complement. This completes the proof.
\end{proof}

\subsection{Construction of Special Basic Block in $\ss$}

Given $A_{m\times m}$, a non-negative irreducible integer matrix which is not a permutation matrix, our aim is to construct a block for the suspension of $\sigma(A)$ so that it verifies the conditions in Proposition~\ref{le7}.

Now suppose that $A_{m\times  m}$ is a matrix with $k$ ones on the diagonal and all other entries equal to zero. Then one has the following proposition.
\begin{proposition}\label{pr9}
Let $v$ be a vertex of $L$ labelled with the suspension of a subshift of finite type 
$\sigma(A)$, where $A_{m\times m}$ is a matrix with $k$ ones on the diagonal and other entries equal to zero. Let $e^+$ and $e^-$ be the indegree and outdegree of $v$ and $\left\{g^+_i\right\}$ and $\left\{g^-_j\right\}$ are the weights on the incoming and outgoing edges incident to $v$ respectively. Suppose that
\begin{center}
$k-\sum g^-_i= e^+\quad\text{and}$\\
$k-\sum g^+_i= e^-$
\end{center}
Then, there exists a basic block $X$ in $\ss$ for the suspension of $\sigma(A)$ with $e^+\,(e^-)$ entering (exiting) boundary components each being a $2$-manifold of genus $g^+_i\,(g^-_j)$. Furthermore, $\ss-X$ is homeomorphic to $H^{g^+_1}_1\sqcup\dots\sqcup H^{g^+_{e^+}}_{e^+}\sqcup H^{g^-_1}_1\sqcup\dots\sqcup H^{g^-_{e^-}}_{e^-}$. 
\end{proposition}
\begin{proof}The dynamics associated to the suspension of $\sigma(A)$ is 
simple. It consists of a basic set with $k$ periodic orbits. One can rewrite the hypothesis
\begin{center}
$(k-1)+1-\sum g^-_i= e^+$,\\
$(k-1)+1-\sum g^+_i= e^-$.
\end{center}

By Proposition $8.2$ in $\cite{de1987smale}$ there exists a way of gluing $k-1$ round handles to $e^-$ collarings of $\partial M_1\dots,\partial M_{e^-}$, where each handlebody $M_i$ has genus $g^-_i$, such that the resulting $3-$manifold $X_1$ is a basic block which contains $k-1$ periodic orbits. Also, $\partial X^+_1(\partial X^-_1)$ is composed by $e^+(e^-)$ closed surfaces of genus $g^+_i(g^-_j)$ respectively and the complement in $\mathbb{S}^3$ is homeomorphic to $H^{g^+_1}_1\sqcup\dots\sqcup H^{g^+_{e^+}}_{e^+}\sqcup H^{g^-_1}_1\sqcup\dots\sqcup H^{g^-_{e^-}}_{e^-}$. Recall that, since $X_1$ is a compact $3$-manifold in $\mathbb{S}^3$, then one can embed $X_1$ in $\ss$. For simplicity we continue to denote this embedding by $X_1$. Now, glue $X_1$ to $R_1$ in order to construct a new basic block $X_2=X_1\cup R_1$ with $k$ periodic orbits. By Lemma~\ref{le18} and Corollary~\ref{co3}, $X_2$ verifies $\partial X_2\approx\partial X^+_1\cup\partial X^-_1$ and $\ss - X_2$ is homeomorphic to $H^{g^+_1}_1\sqcup\dots\sqcup H^{g^+_{e^+}}_{e^+}\sqcup H^{g^-_1}_1\sqcup\dots\sqcup H^{g^-_{e^-}}_{e^-}$. At this point,  a basic set with $k$ periodic orbits corresponding to the $k$ ones on the diagonal of $A$ has been constructed. Also, $m-k$ nilpotent handles must be glued to $X_2$. Recall that these handles are homeomorphic to $D^1\times D^1\times D^1$ and contain no recurrent points. These handles will correspond to the $m-k$ zeros on the diagonal of $A$. Finally, denote $X$ as $X_2$ glued with $m-k$ nilpotent handles. Recall that these nilpotent handles do not modify $\partial X_2$. 
    
\end{proof}
For the general case, we can construct a basic block for the suspension of $\sigma(N)$ where $N$ is a matrix that is flow equivalent to $A$. By definition, this implies that the suspensions of $\sigma(A)$ and $\sigma(N)$ are topologically equivalent. An essential proposition at this point is due to Franks~\cite{franks1985nonsingular}.
\begin{proposition}\label{pr8}
If $A$ is a non-negative irreducible integer matrix which is not a permutation matrix then given an integer $M>0$ there is a matrix which is flow equivalent to $N$ and which has every entry greater than $M$ and non-diagonal entries even. The size of $N$ depends only on $A$ and not on $M$.
\end{proposition}
By Proposition~\ref{pr8}, the matrix $A_{m\times m}$ is flow equivalent to some matrix $N_{n\times n}$ which has all non-diagonal entries even, 
\begin{center}
$\dim \ke\left( (I-\overline{N}:F^n_2\rightarrow F^n_2)\right)=\dim \ke\left( (I-\overline{A}:F^n_2\rightarrow F^n_2)\right)=k,$
\end{center}
where $\overline{N}$ denotes the $\mo 2$ reduction of $N$. By Proposition~\ref{pr9} we are able to construct a basic block $X$ for the suspension of $\sigma(\overline{N})$. Our goal is to construct a basic block for the suspension of $\sigma(N)$. In order to achieve this, we wish to maintain the block $X$ while modifying the flow within it. We will use the notion of one-handles $H_i$ within round handles $R_i$ and nilpotent handles $N_t$. 

Let $x$ be a point of $D^1\times p\subset H_i$ where $p\in D^1$. The interval $D^1\times p$ will be denoted by $W^u_i(x)$. Similarly, let $x$ be a point of $q\times D^1$. The interval $q\times D^1$ will be denoted by $W^s_i(x)$. The reason for introducing one-handles is that if $H_i$ intersects $m$ times $\mu(H_j)$  then the $ij$th entry of $\overline{N}$ increases by $m$, where $\mu(x)$ is the first return map for the one-handle of the round handle. 

This is precisely what we want to do, namely, increase the entries of $\overline{N}$ to obtain $N$. For this reason we will construct a connected surface $U$, such that, $U$ contains all one handles and the flow is transversal to $U$.

For each one handle $H_i$, we consider two disjoint subsets $E^i_1$ and $E^i_2$ contained in $H_i$ as shown in Figure~\ref{figure18}, homeomorphic to $D^1\times D^1$ such that 
$$
\mu(E^i_1\sqcup E^i_2)\cap H_i=\emptyset,
$$
we will call $E^i_1$ and $E^i_2$ the \textbf{\textit{set of ends}} of $H_i$. See Figure~\ref{figure18}.
\begin{figure}[h!]
\centering{\includegraphics[width=8cm]{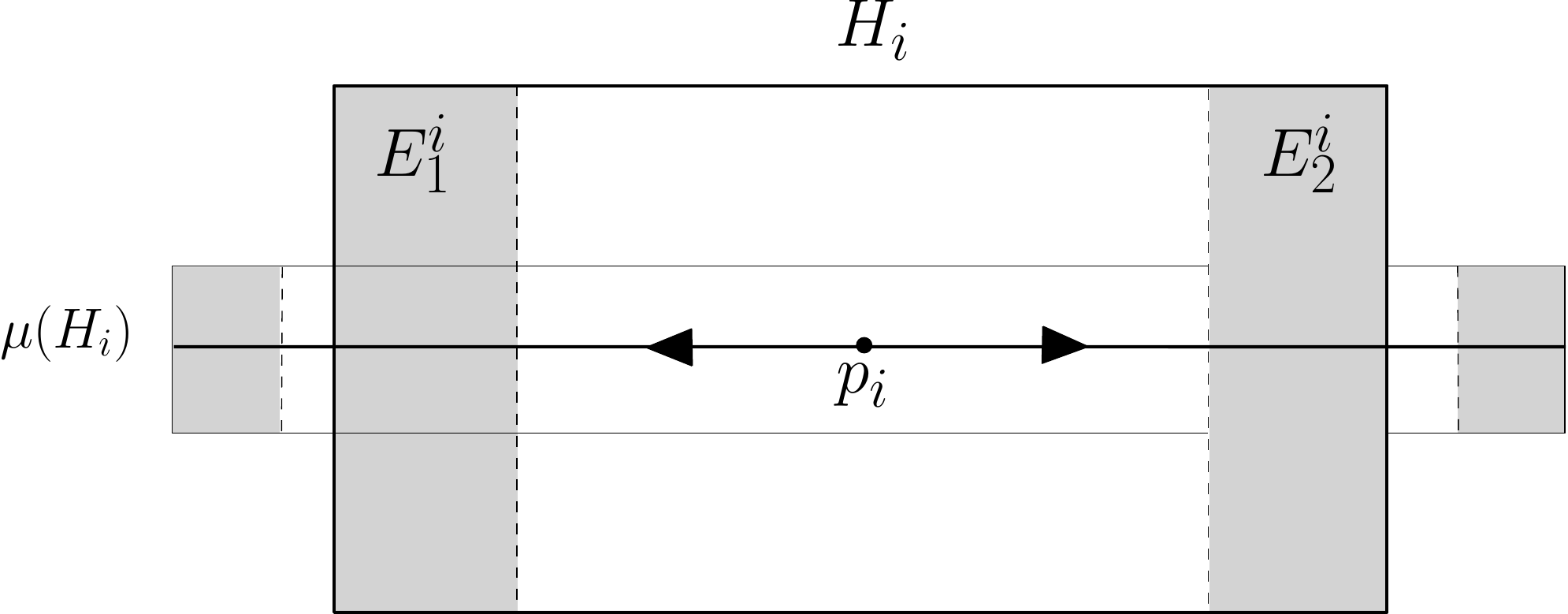}}
\caption{$p_i$ is the center of the one handle $H_i$}\label{figure18}
\end{figure}

Thus, our aim is the construction of a surface $U\subset X$ such that $U$ contains all one handles $H_i$ and all nilpotent handles $N_j$. 

On the other hand, in order to connect the one handles we consider disjoint rectangles embedded in the union of collaring of $\partial M_i$ for $i\in\left\{1,\dots,e^-\right\}$ and consider the flow transversal to the collaring as described in Proposition~\ref{pr10}.
We will use these rectangles to connect the ends of the one handles associated to the collaring. 

First fix a collaring of a handlebody $\partial M_i\times I$ for $i\in \left\{1,\dots,e^-\right\}$. Now consider the round handles $R_{j_1},\dots,R_{j_l},\dots,R_{j_{n(i)}}$ which glue to $\partial M_i\times I$. Without loss of generality, we can suppose that $R_{j_1},\dots,R_{j_l}$ are round handles, that have only one annulus which glues to $\partial M_i\times I$ and the remaining ones are invariant round handles with two annuli which glue to  $\partial M_i\times I$. The special round handle $R_1$ shall be considered as if $R_1$ has one annulus glued to $\partial M_1\times I$. Now, we connect the ends of ones handles $H_{j_1},\dots,H_{j_l}$ with a rectangle embedded in the collaring of $\partial M_i$ as described above. Thus, we can find a connected surface $U^i_1$ such that, $U^i_1$ contains all one handles $H_{j_1},\dots,H_{j_l}$ of the round handles $R_{j_1},\dots,R_{j_l}$ glued to $\partial M_i\times I$. For the invariant round handles, we glue a rectangle to $E^{j_{l+1}}_1$ and another rectangle to $E^{j_{n(i)}}_2$. Now, we connect $H_{j_{l+1}}$ with $H_{j_{l+2}}$ by gluing a rectangle joining $E^{j_{l+1}}_2$ and $E^{j_{l+2}}_1$. We do this successively until the last pair $H_{j_{n(i)-1}}$ and $H_{j_{n(i)}}$ have been joined. Thus, we can find a connected surface $U^i_2$, such that $U^i_2$ contains all one handles $H_{j_{l+1}},\dots,H_{j_{n(i)}}$. Now we connect $U^i_1$ and $U^i_2$ with a rectangle to obtain a connected surface $U^i$. Note that the flow is transversal to $U^i$.

In this fashion, we construct a collection of $\mathcal{U}=\left\{\,U_i\,|\,\, 1\leq i\leq e^-\right\}$. Note that, $U^i\cap U^j$ is composed by the one handles that connect the collaring of $\partial M_i$ with the collaring of $\partial M_j$. Define
$$
U_1=\bigcup^{e^-}_{i=1}U^i.
$$
This surface $U_1$ contains all one handles associated to $R_1,\dots,R_k$. Extend $U_1$ by a rectangle $\hat{R}$, which contains all nilpotent one handles. See Figure~\ref{figure22}.
\begin{figure}[h!]
\centering{\includegraphics[width=8cm]{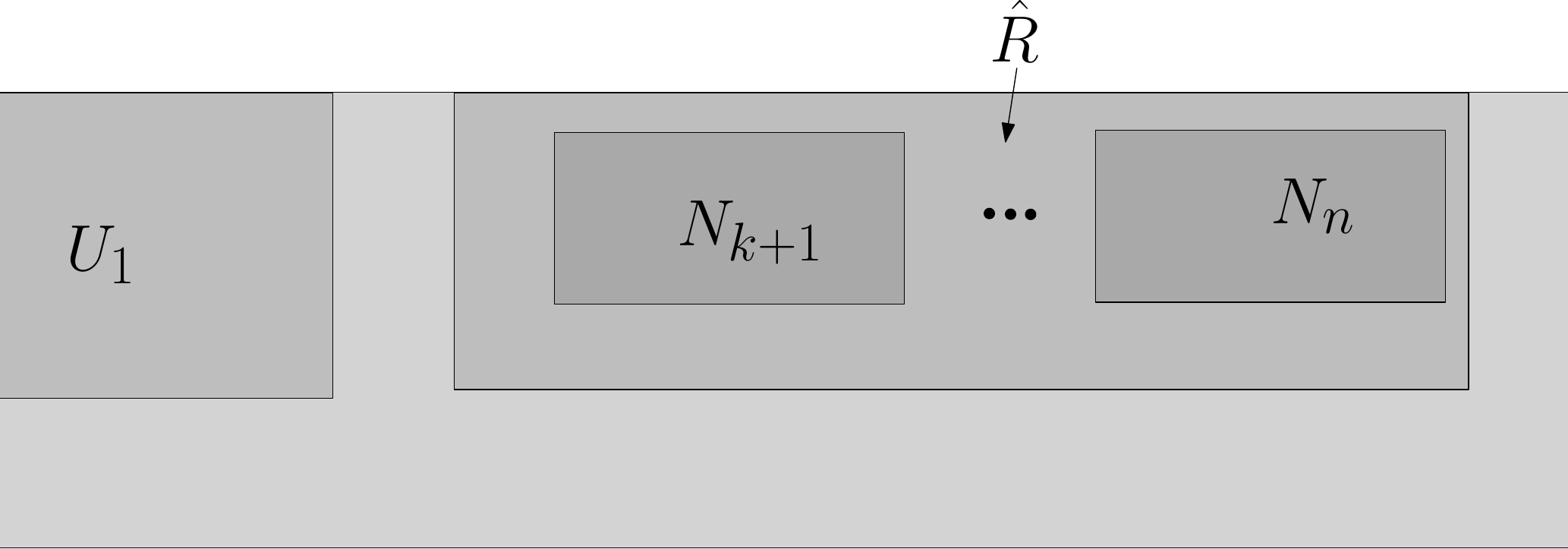}}
\caption{$\hat{R}$ contains all nilpotent one handles.}\label{figure22}
\end{figure}

Thus, by Proposition~\ref{pr10} one can find a surface $U$ that contains $U_1$ and $\hat{R}$. Actually, one finds a surface $V$, such that $U\subset V$  and the first return map $\mu:U\rightarrow V$ is smooth. By construction, one has that
$$
\overline{U}=(U-\mu(U))\cup \bigcup^{e^-}_{i=1}\mu(E^i_1\sqcup E^i_2)
$$ 
is connected. 
\begin{lemma}\label{le17}
Let $U_0$ be a small neighborhood of $U$ in $V$ as defined above. There exists an isotopy $\Psi_t:V\rightarrow V$ supported on the interior of $U_0$ such that $\Psi_0$ is the identity. Also, $\mu_1=\Psi_1\circ \mu:U\rightarrow V$ satisfies that its suspension flow with induced flow $\varphi_t$ has a chain recurrent set which is topologically equivalent to the suspension of a subshift of finite type with matrix $A_{m\times m}$. 
\end{lemma}
\begin{proof}
We have that $\overline{U}$ is connected. Hence, for $p\in \inte(\mu(W^u_i(p_i)\cap (\partial M_i\times I))$ and $q\in\inte H_j$ such that $W^u_j(q)$ and $\mu(H_\alpha)$ are disjoint $\forall \alpha\in\left\{1,\dots,e^-\right\}$, there exists a curve $\gamma:[0,3]\rightarrow \inte\overline{U}$  such that $\gamma(1)=p$ and $\gamma(2)=q$. Also, the intersection of $\mu(H_i)$ and $\gamma([0,3])$ is connected and equal to $\mu(W^s_i(\mu^{-1}(p)))$ and its intersection with $H_j$ lies in $W^u_j(q)$. For the other $H_l$'s with $l\in\left\{1,\dots,n\right\}-\left\{i,j\right\}$ and $\gamma([0,3])\cap H_l\neq\emptyset$, $\gamma$ is transversal to $W^u_l(p_l)$. See Figure~\ref{figure19}. 
\begin{figure}[h!]
\centering{\includegraphics[width=11cm]{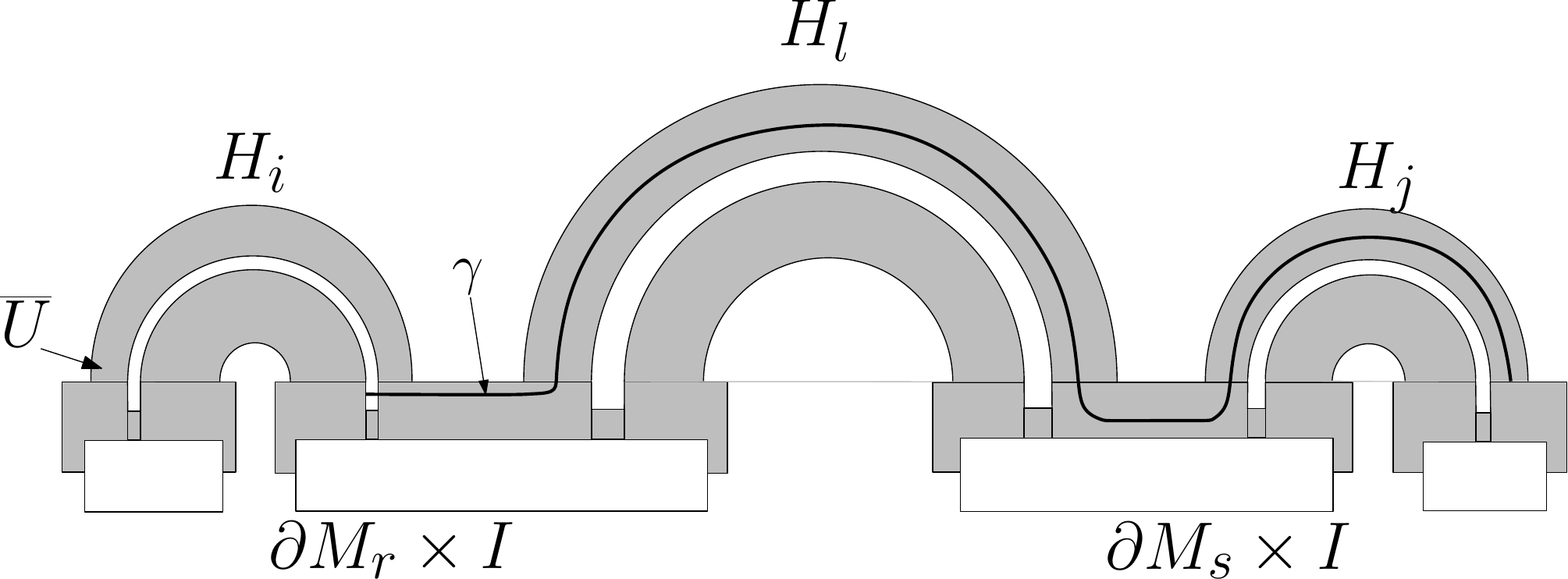}}
\caption{$H_l$ is a connecting handle.}\label{figure19}
\end{figure}

Thus, there is an isotopy supported in a tubular neighborhood of $\gamma([0,1])$ which pushes a small interval of $\mu(W^u_i(p_i))$ along the curve $\gamma$ until it intersects $W^s_j(p_j)$ in two points. It is possible that this process causes intersection with other $W^s_l(p_l)$ but always an even number of them. See Figure~\ref{figure20}.
\begin{figure}[h!]
\centering{\includegraphics[width=9cm]{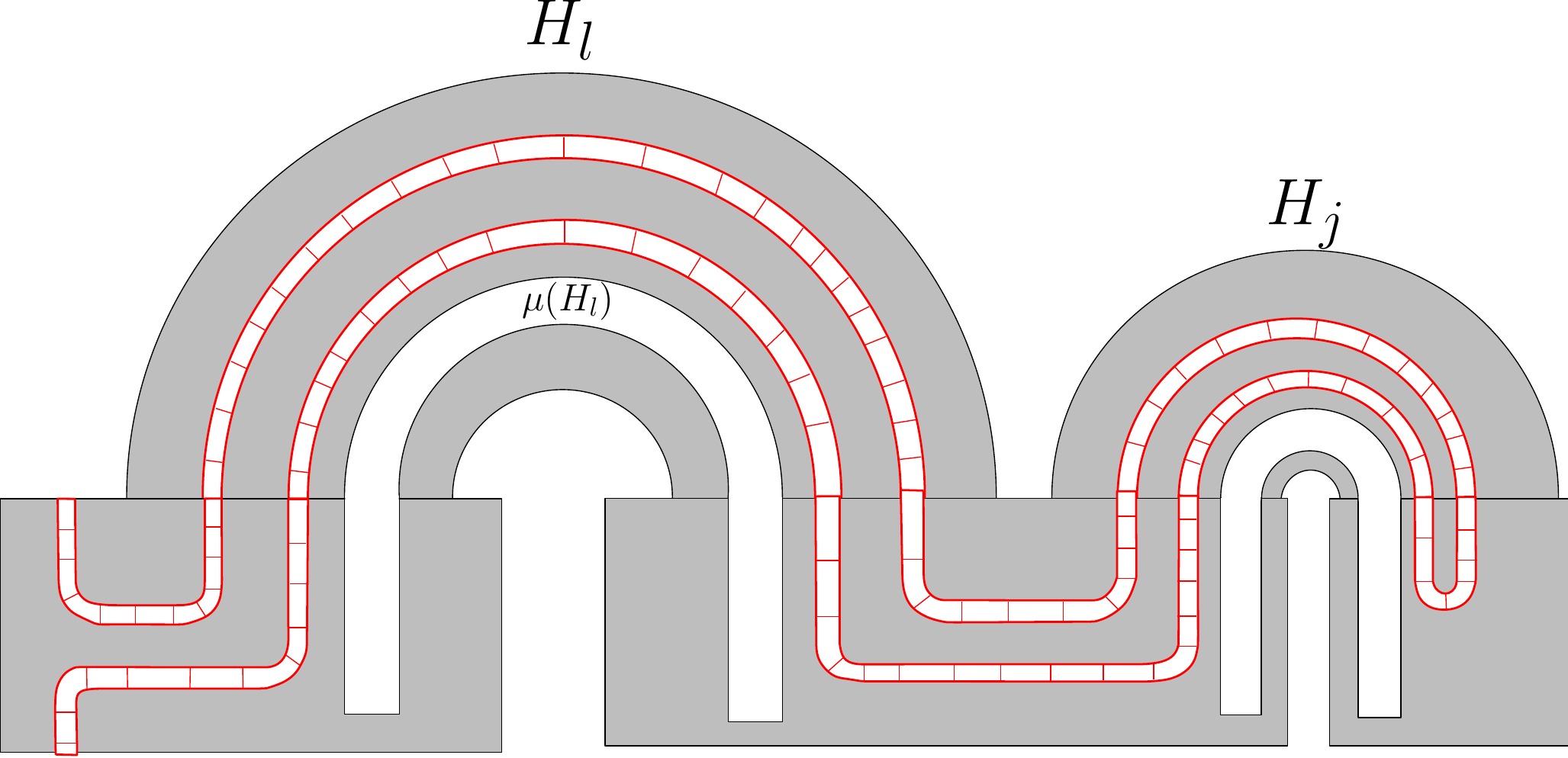}}
\caption{Deformation of $\phi_1(H_i)$.}\label{figure20}
\end{figure}

Note that this isotopy preserves the hyperbolicity of the one handles with respect to the first return map $\phi_1$. In others words, if $\mu(W^u_i(x))\cap W^u_j(y)\neq \emptyset$ then $W^u_j(y)\subset \mu(W^u_i(x))$ and similarly if $\mu(W^s_i(x))\cap W^s_j(y)\neq \emptyset$ then $\mu(W^s_i(x))\subset W^s_j(y)$ $\forall x\in H_i$ and $y\in H_j$.\\
Hence, to achieve the geometric intersection matrix $N$, which is flow equivalent to $A$, we argue as in \cite{franks1985nonsingular}. First one needs to find a pairwise disjoint family of embedded curves $\gamma_{ij}:[0,3]\rightarrow \overline{U}$ with the properties described above for $i,j\in\left\{1,\dots,n\right\}$. The existence of this family $\left\{\gamma_{ij}\right\}$ with these properties is easy to see except for the fact that they are  disjoint. Now, by using an isotopy supported on a neighborhood of $\gamma_{11}$ we push off all other curves. We do this for $\gamma_{12}'$ and in the same way for all the remaining curves.\\
By Proposition~\ref{pr8} we can, if necessary, replace the matrix $N$ for another matrix which is flow equivalent to it, congruent $\mo 2$ and has every entry as large as we want. In particular larger than the corresponding entry of $A$. We can suppose this has been done and for simplicity of notation continue to call this matrix $N$. Now for each $\gamma_{ij}$ we can increase the intersection points of $W^s_j(p_j)$ with the image of $W^u_i(p_i)$ if we push back the curve through $H_j$ again. See Figure~\ref{figure21}.
\begin{figure}[h!]
\centering{\includegraphics[width=9cm]{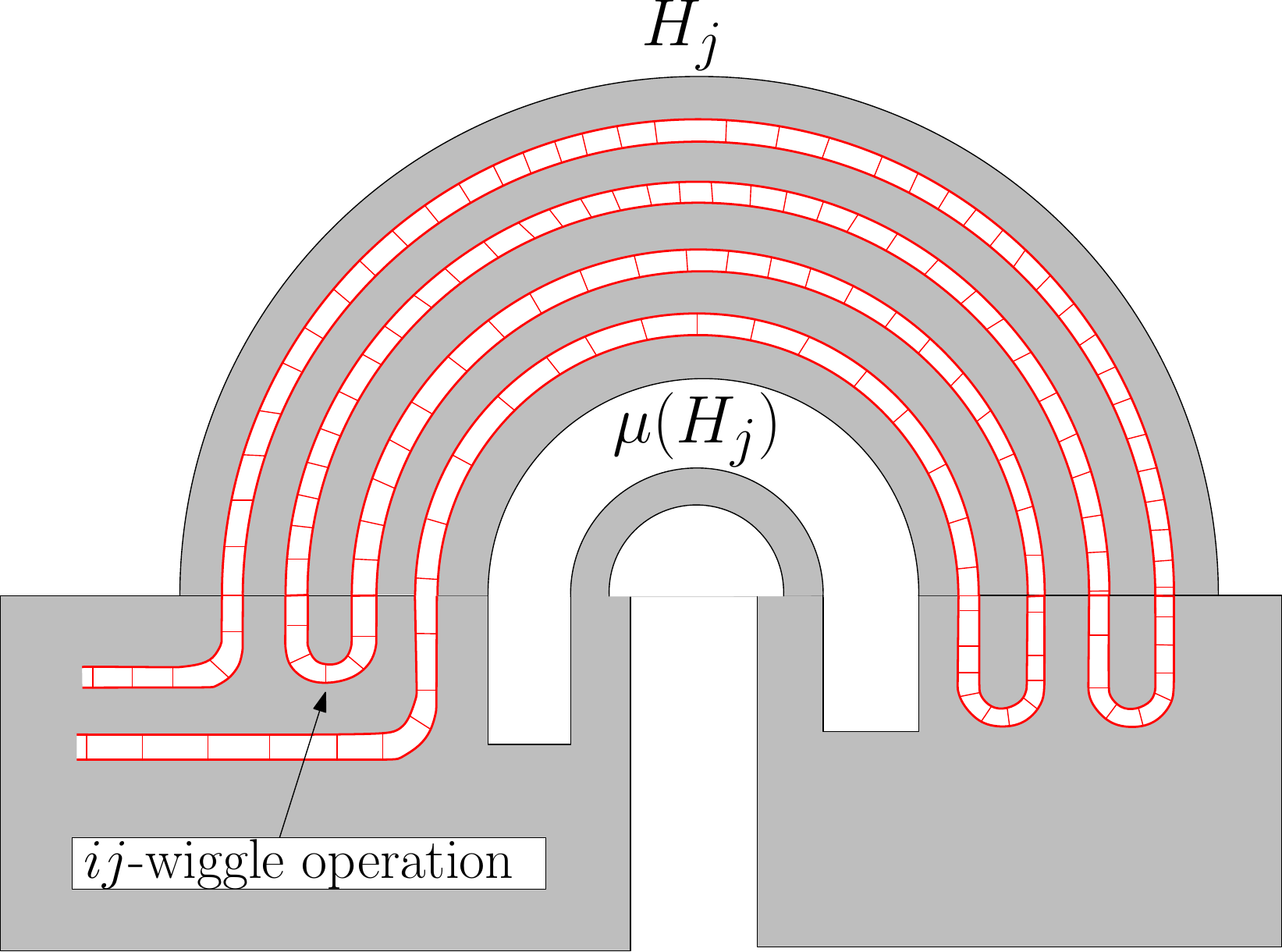}}
\caption{$ij-$wiggle operation.}\label{figure21}
\end{figure}
\\
This is referred to as the $ij-$wiggle operation. If this operation is repeated $r$ times, $2r+1$ intersection points are added. Do  this with $r=\frac{1}{2}[(A)_{ij}-(N)_{ij})]$ so that the $ij$ entry of the matrix $\overline{N}$ agrees with $(N)_{ij}$. Repeat this process in order to get to the matrix $N$. Now isotope $\mu$ so that the final map has the desired  property.
\end{proof}
\begin{proposition}\label{pr-suff} Let $A_{m\times m}$ be a non-negative integer matrix with $k=\dim \ke(\overline{(I-A)}F_2^m\rightarrow F_2^m$, where  $\overline{A}$ is the $\mo 2$ reduction of $A$. Suppose also that $e^+,e^-,g^+_j,g^-_i\in \N$, with $i=1,\dots,e^+$ and $j=1\dots,e^-$, are positive integers satisfying
\begin{enumerate}
\item $e^--e^+=\sum g^-_i-\sum g^+_j$,
\item $k-\sum g^-_i= e^+$ and
\item $k-\sum g^-_i= e^+$.
\end{enumerate}
Then there exists a Smale flow $\phi_t$ on $\ss$ with a basic block $X$ such that
\begin{enumerate}
\item The flow $\phi_t$ restricted to the basic set $\Lambda\subset X$ is topologically equivalent to the suspension of $\sigma(A)$, and
\item $\partial X^+(\partial X^-)$ has $e^+(e^-)$ components composed by surfaces of genus $g^+_i(g^-_j)$.
\item $(\ss)-X$ is homoemorphic to $H^{g^+_1}_1\sqcup\dots\sqcup H^{g^+_{e^+}}_{e^+}\sqcup H^{g^-_1}_1\sqcup\dots\sqcup H^{g^-_{e^-}}_{e^-}$ 
\end{enumerate}
\end{proposition}
\begin{proof}

By Proposition \ref{pr8}, one knows that $A_{m\times m}$ is flow equivalent to $N_{n\times n}$. By Proposition~\ref{pr9}, we can build a basic block $X_1$ for $\overline{N}$,
which satisfies assertions $2$ and $3$ of this proposition. For the first assertion, by Lemma~\ref{le17}, one has that, for $U_0$ chosen sufficiently small, $\mu:U_0\rightarrow V$ is smooth and $\tau: U_0\rightarrow\re$ is the smallest $t>0$ such that $\phi_t(x)=\mu(x)$. Then the partial flow on
$$
Z=\left\{\phi_t(x)|x\in U_0,\, 0\leq t\leq \tau(x)\right\}.
$$
is the suspension flow for $\mu$. Now by the construction, the suspension flow for $\mu_1$ is the same as for $\mu$ since $\mu$ and $\mu_1$ are isotopic. On the other hand, since the isotopy is supported on the interior of $U_0$, one has that $\mu$ and $\mu_1$ agree near the boundary of $U_0$. Then near the boundary of $Z$ the suspension flow of $\mu_1$, $\varphi_t$ and the suspension flow of $\mu$, $\phi_t$ agree. Now, let $\eta_t$ be a flow on $\ss$ which is generated by the vector field which is tangent to the suspension flow of $\mu_1$ on $Z$ and tangent to $\phi_t$ elsewhere. Hence $\eta_t$ is a Smale flow which satisfies  this proposition.
\end{proof}

\textbf{Proof of the sufficiency of the conditions of Theorem~\ref{MainTheorem}.}

The sufficient conditions of  Theorem~\ref{MainTheorem} will be presented in two parts $\beta(L)=1$ and $\beta(L)=0$.\\

\textbf{Case} $\beta(L)=1$:\\

Let $L$  be an abstract Lyapunov graph with a cycle that satisfies the conditions (1), (2), (3a) and (4). Suppose $a$ is an edge in the cycle of $L$ with weight $g$. Consider the graphs $L_1$ and $L_2$ each with a dangling edge as shown in Figure~\ref{figure8}.
\begin{figure}[h!]
\centering\includegraphics[scale=0.5]{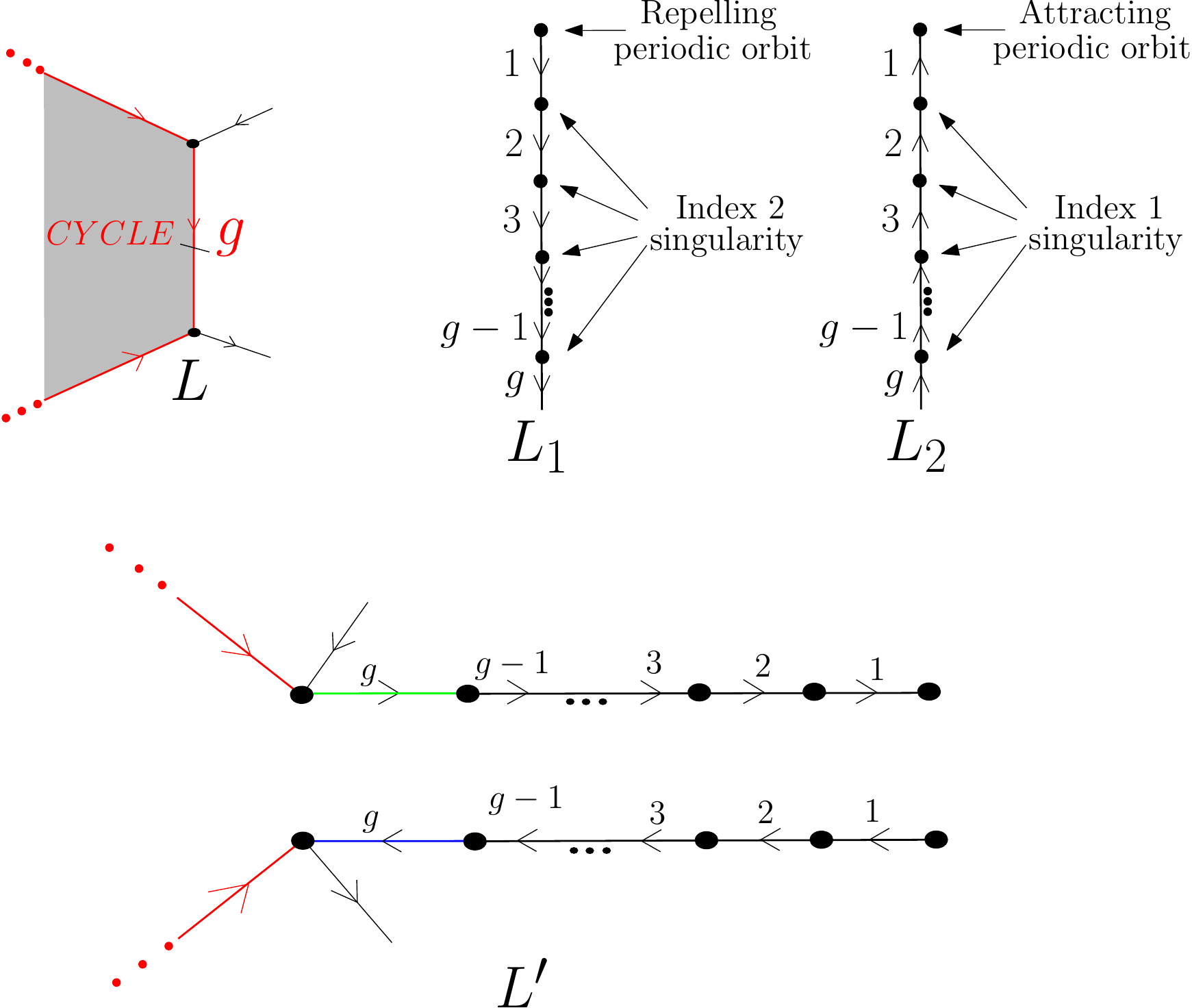}
\caption{The graphs $L_1$ and $L_2$ have a dangling edge.}\label{figure8}
\end{figure}

Now cut $L$ along $a$ and glue the graphs $L_1$ and $L_2$ by the dangling edges, as shown in Figure~\ref{figure8}. Then a new abstract Lyapunov graph $L'$ is obtained such that, $L'$ is a tree and each vertex satisfies the conditions in Theorem~\ref{teorem1}. Therefore, there exists a Smale flow $\psi_{t}$ on $S^{3}$ with Lyapunov graph $L'$ and such that $L_1$ and $L_2$ are associated to two handlebodies $H^g_1$ and $H^g_2$, whose boundaries are unlinked  in $\mathbb{S}^3$, as shown in Figure~\ref{figure4}.

\begin{figure}[h!]
\centering\includegraphics[scale=0.5]{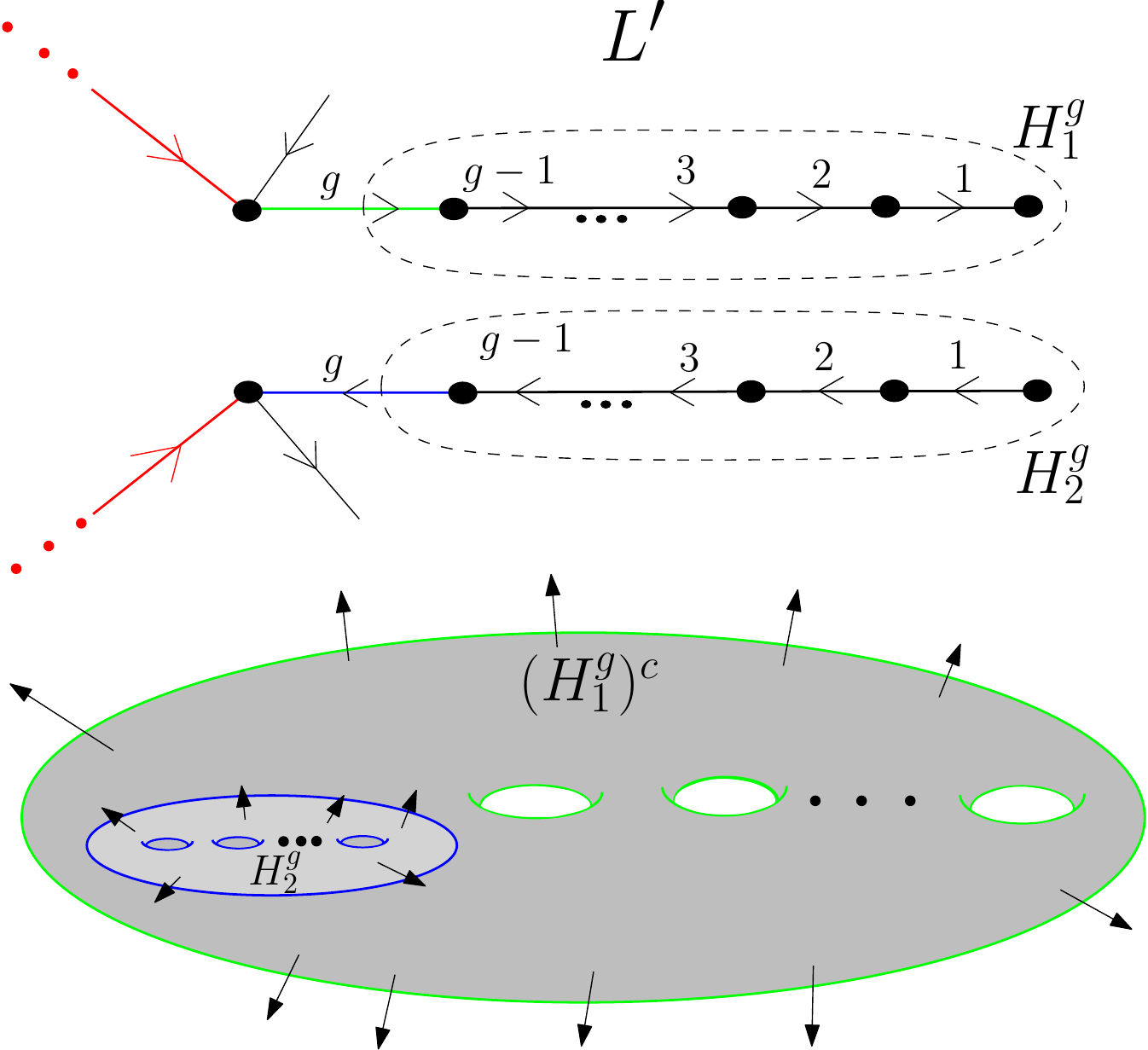}
\caption{$L'$ corresponds to a Smale flow on $\mathbb{S}^3$}.\label{figure4}
\end{figure}

 Now we cut the two neighborhoods $N(H^g_1)$ and $N(H^g_2)$ of $H^g_1$ and $H^g_2$ respectively. Then gluing $S^{3}-\left(N(H^g_1)\sqcup N(H^g_2)\right)$ along $\partial N(H^g_1)$ and $\partial N(H^g_2)$ suitably, we obtain a Smale flow $\varphi_t$ on $\ss$ with Lyapunov graph $L$.\\

\textbf{Case} $\beta(L)=0$: In this case we will consider two possibilities: (3b)i and (3b)ii.\\

Let $L$ be an abstract Lyapunov tree with satisfies (1), (2), (3b)ii and (4).\\ 

By the condition (3b)ii, there is an edge $a$ of $L$, such that the weight of $a$ is $g>0$. Now we cut $L$ along $a$ and glue the graphs $L_1$ and $L_2$ by the dangling edges as shown in Figure~\ref{figure25}. 

\begin{figure}[h!]
\centering\includegraphics[scale=0.6]{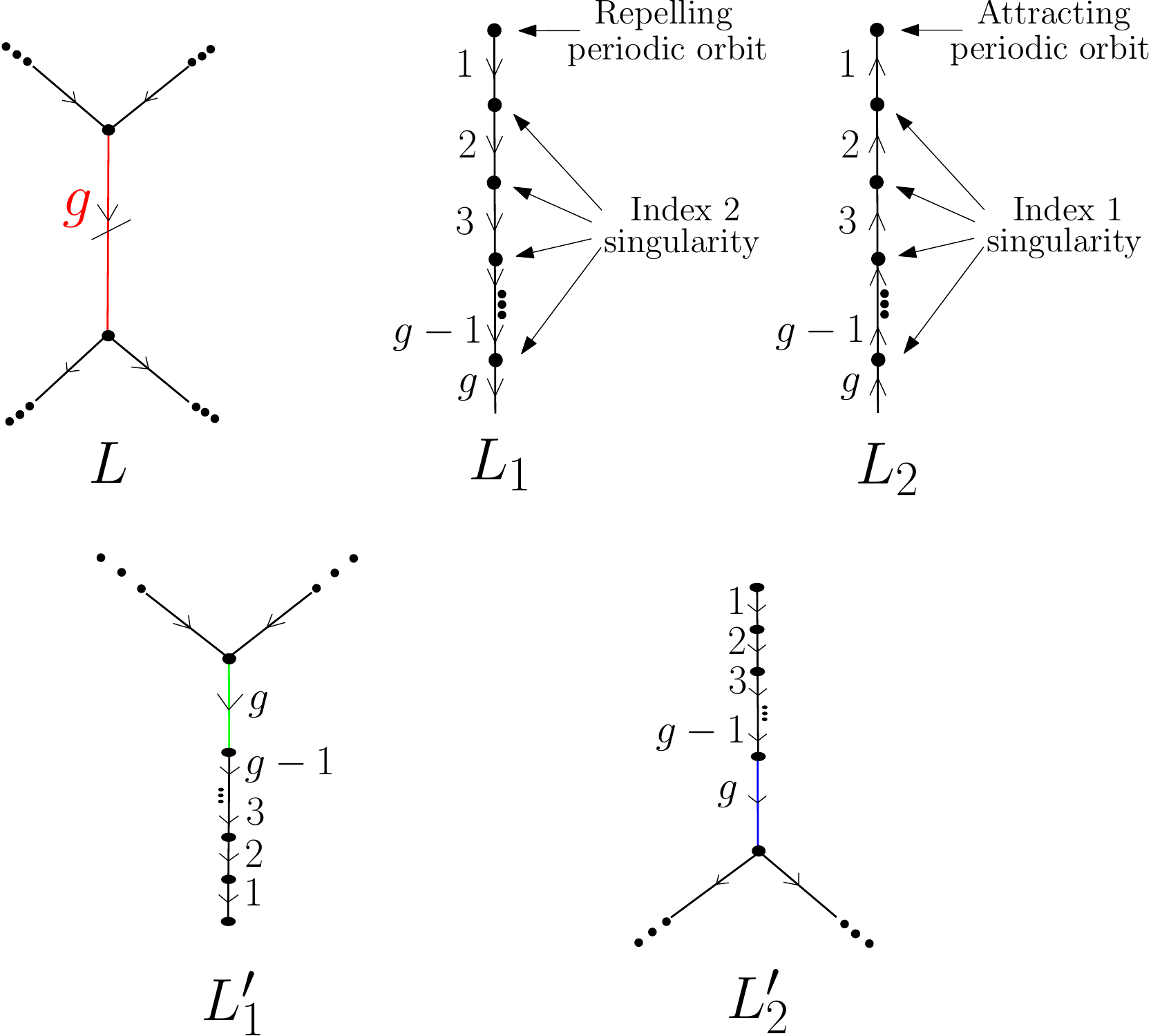}
\caption{The graphs $L_1$ and $L_2$ have dangling edge.}\label{figure25}
\end{figure}

The new graphs $L_1^\prime$ and $L_2^\prime$ are obtained. Since, $L_1^\prime$ and $L_2^\prime$ satisfies the conditions of Theorem~\ref{teorem1}, there exists Smale flows $\psi_t$  and $\phi_ t$ on $\mathbb{S}^3$, such that the graph  $L_1$ corresponds to a handlebody manifold $H^g_1$ and the graph  $L_2$ corresponds to a handlebody manifold $H^g_2$, as shown in Figure~\ref{figure26}. 

\begin{figure}[h!]
\centering\includegraphics[scale=0.6]{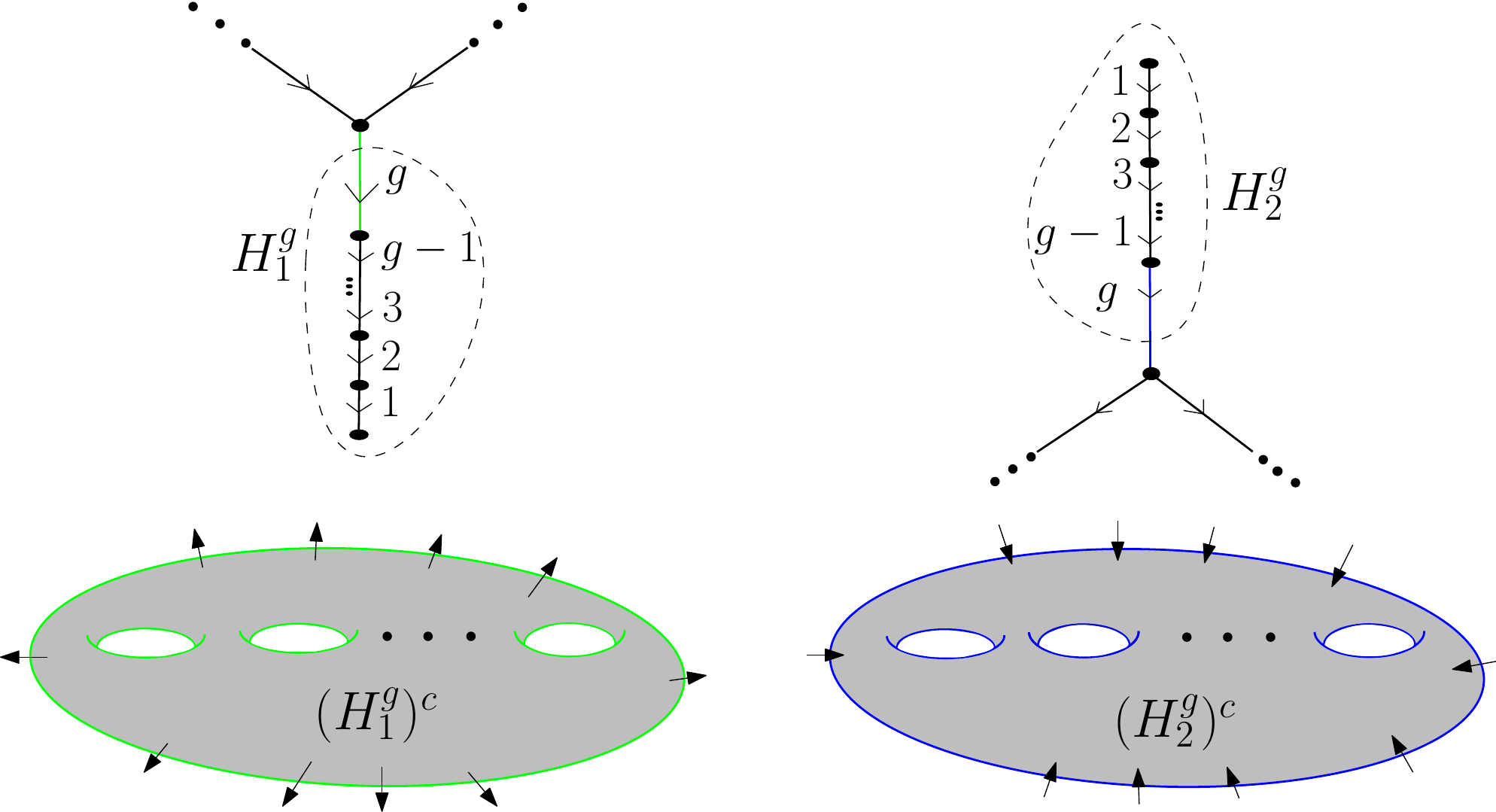}
\caption{The graphs $L'_1$ and $L'_2$ correspond to Smale flows on $\mathbb{S}^3$.}\label{figure26}
\end{figure}

This implies that $L_2$ corresponds to a handlebody manifold $H^g_2$ as well. 
Then we glue $((H^g_1)^c,\psi_t|_{(H^g_1)^c})$ and $((H^g_2)^c,\psi_t|_{(H^g_2)^c})$ suitably along their boundaries. We obtain a Smale flow $\xi_t$ on $\ss$ with Lyapunov graph $L$.\\

Let $L$ be an abstract Lyapunov tree which satisfies (1), (2), (3b)i and (4).\\

Let $v$ be the unique vertex that verifies the first condition of (3b)i. By Proposition~\ref{pr-suff}, there exists a Smale flow $\phi_t$ on $\ss$ and a basic block $X$ embedded in $\ss$, associated to the vertex $v$. Also, $\ss-X$ is homeomorphic to $H^{g^+_1}_1\sqcup\dots\sqcup H^{g^+_{e^+}}_{e^+}\sqcup H^{g^-_1}_1\sqcup\dots\sqcup H^{g^-_{e^-}}_{e^-}$. Now, if we cut $L$ along all incoming and outgoing edges incident to the vertex $v$, we obtain $e^++e^-$ subgraphs with dangling edges which can be denoted by $L^+_1,\dots,L^+_{e^+}$ and $L^-_1\dots,L^-_{e^-}$, as shown in Figure~\ref{figure27}.
\begin{figure}[h!]
\centering{\includegraphics[width=9cm]{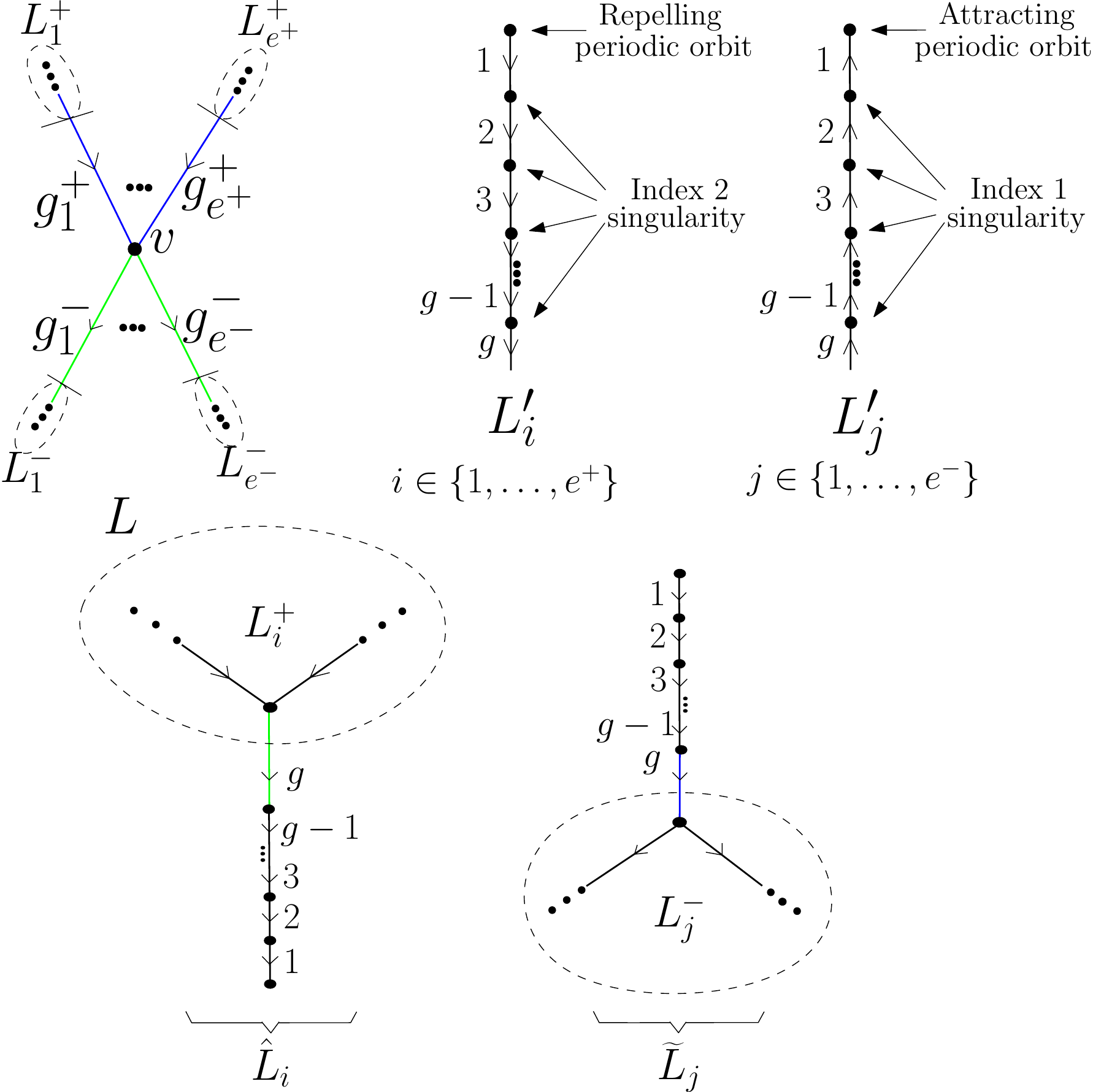}}
\caption{The graphs $L'_i$ and $L'_j$ have a dangling edge.}\label{figure27}
\end{figure}
 We use the graphs $L'_i$ and $L'_j$ with dangling edges to create the new graphs $\hat{L}_1,\dots,\hat{L}_{e^+}$ and $\widetilde{L}_1\dots,\widetilde{L}_{e^-}$. Then by Proposition~\ref{pr5} and Theorem~\ref{teorem1}, there exist Smale flows $\phi^i_t$ and $\varphi^j_t$ on $\mathbb{S}^3$ such that $L^+_i$ and $L^-_j$ correspond to handlebodies $H^+_i$ and $H^-_j$ respectively, as shown in Figure~\ref{figure28}.
\begin{figure}[h!]
\centering{\includegraphics[width=10cm]{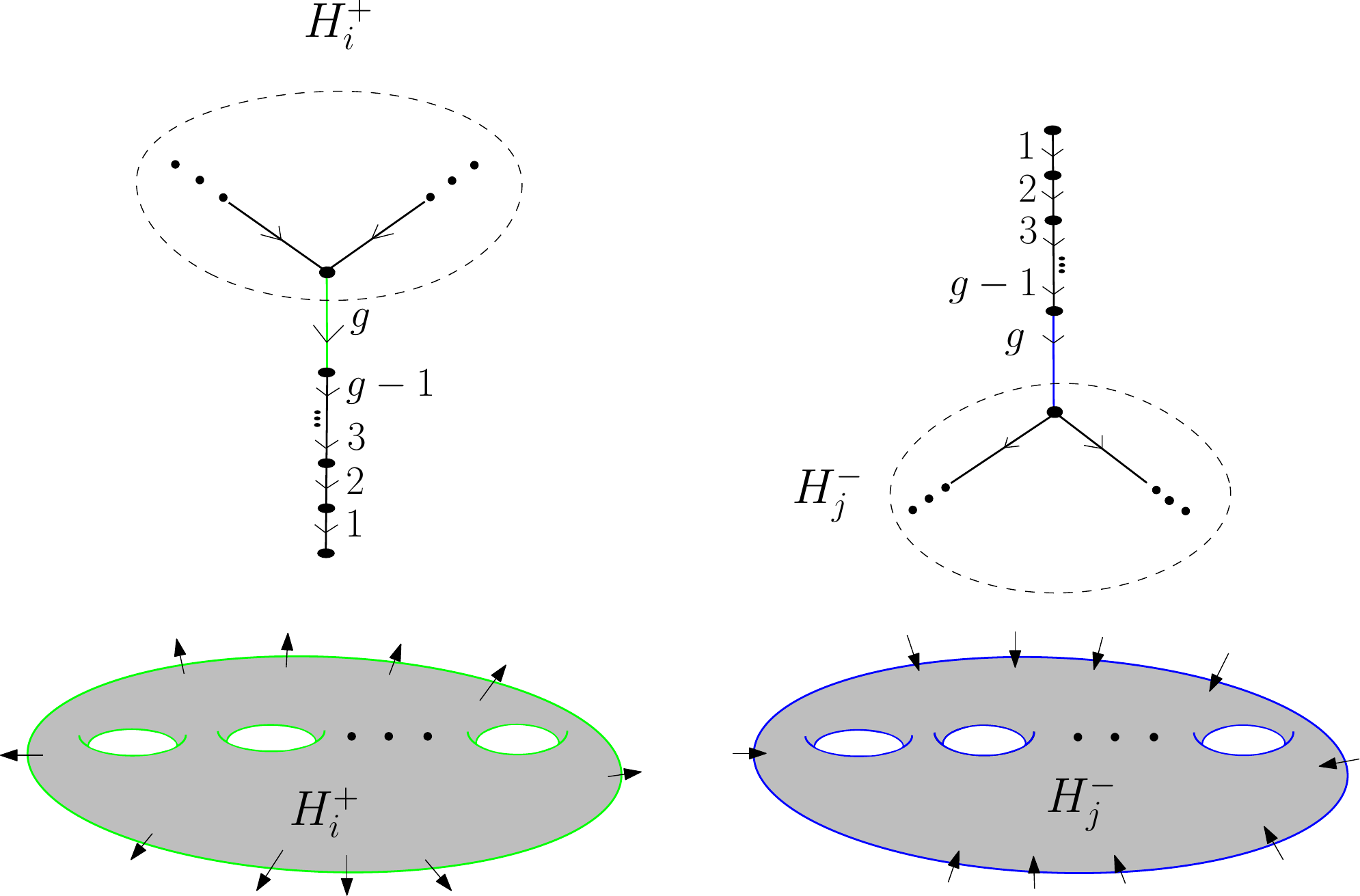}}
\caption{The graphs $\hat{L}_i$ and $\widetilde{L}_j$ correspond to Smale flows on $\mathbb{S}^3$.}\label{figure28}
\end{figure} 
  Then we have constructed Smale flows on $(X,\phi_t|_X)$, $(H^+_i,\phi^i_t|_{H^+_i})$ and $(H^-_j,\varphi^i_t|_{H^-_j})$ for $i\in\left\{1,\dots,e^+\right\}$ for $j\in\left\{1,\dots,e^-\right\}$. Finally, we glue these manifolds suitably, in order to obtain a Smale flow on $\ss$ with Lyapunov graph $L$. 

\cqd
  
\begin{center}{\sc acknowledgments}
\end{center}

The authors would like to thank the referee for reading the paper carefully and providing many helpful suggestions in organizing the layout of the main result and consequently improving its readability.

\vspace{0.3cm}

{\sc Department of Mathematics, Institute of Mathematics,
Statistics and Scientific
Computation, Unicamp, Campinas, S\~ao Paulo, Brazil

\vspace{0.3cm}

\sc Institute of Sciences Mathematics and of Computation,
University Of S\~ao Paulo, S\~ao Carlos, S\~ao Paulo, Brazil

}

\end{document}